\documentclass[a4paper]{scrartcl}

\usepackage{lmodern}

\usepackage[T1]{fontenc}
\usepackage[utf8]{inputenc}

\usepackage{amssymb}
\usepackage{amsmath}
\usepackage{amsthm}

\usepackage{hyperref}
\usepackage[capitalise]{cleveref}

\usepackage{xcolor}

\theoremstyle{plain}
\newtheorem{theorem}{Theorem}[section]
\newtheorem*{theorem*}{Theorem}
\newtheorem{proposition}[theorem]{Proposition}
\newtheorem{lemma}[theorem]{Lemma}
\newtheorem{corollary}[theorem]{Corollary}

\theoremstyle{definition}
\newtheorem{definition}[theorem]{Definition}
\newtheorem{example}[theorem]{Example}
\newtheorem*{convention}{Convention}
\newtheorem{conjecture}[theorem]{Conjecture}

\theoremstyle{remark}
\newtheorem{remark}[theorem]{Remark}
\newtheorem{question}[theorem]{Question}

\newcommand{\Aut}{\mathrm{Aut}}
\newcommand{\Isom}{\mathrm{Isom}}

\newcommand{\Stab}{\mathrm{Stab}}
\newcommand{\chara}{\mathrm{char}}
\newcommand{\id}{\mathrm{id}}

\newcommand{\bbR}{\mathbb{R}}
\newcommand{\bbH}{\mathbb{H}}

\newcommand{\calP}{\mathcal{P}}

\newcommand{\calD}{\mathcal{D}}
\newcommand{\calA}{\mathcal{A}}
\newcommand{\calG}{\mathcal{G}}
\newcommand{\calF}{\mathcal{F}}

\newcommand{\Cen}{\mathrm{Cen}}
\newcommand{\Inv}{\mathrm{Inv}}

\newcommand{\MR}{\mathrm{MR}}
\newcommand{\Mdeg}{\mathrm{MD}}
\newcommand{\fix}{\mathrm{fix}}

\newcommand{\perf}{\mathrm{perf}}
\newcommand{\inv}{{-1}}
\newcommand{\half}{{1/2}}
\newcommand{\ihalf}{{-1/2}}
\newcommand{\squared}{{\cdot 2}}

\newcommand{\PSL}{\mathrm{PSL}}

\renewcommand{\phi}{\varphi}

\newcommand\subsetsim{\mathrel{\substack{
  \textstyle\subset\\[-0.2ex]\textstyle\sim}}}

\begin{document}

\title{Mock hyperbolic reflection spaces and Frobenius groups of finite Morley rank}
\author{Tim Clausen \quad and \quad Katrin Tent}

\maketitle

\begin{abstract}
 We define the notion of mock hyperbolic reflection spaces and use it to study Frobenius groups, in particular in the context of groups of finite Morley rank including the so-called \emph{bad groups}. We show that connected Frobenius groups of finite Morley rank and odd type with nilpotent complement split or interpret a bad field of characteristic zero. Furthermore, we show that mock hyperbolic reflection spaces of finite Morley rank satisfy certain rank inequalities, implying in particular that any connected Frobenius group of odd type and Morley rank at most ten either splits or is a simple non-split sharply $2$-transitive group of characteristic $\neq 2$ of Morley rank $8$ or $10$.
\end{abstract}

\section{Introduction}
This paper contributes to the study of groups acting on geometries arising naturally from conjugacy classes of involutions. We define the notion of a mock hyperbolic reflection space and use it to study certain Frobenius groups. Such an approach to the classification of groups and their underlying geometries based on involutions was developed by Bachmann~\cite{bachmann}.
Mock hyperbolic reflection spaces generalize real hyperbolic spaces and their definition is motivated by the geometry arising from the involutions in certain non-split sharply $2$-transitive groups. 

The points of such a mock hyperbolic space are given by a conjugacy class of involutions and we view the  conjugation action by an involution in the space as a point-reflection. More precisely,
a conjugacy class of involutions in a group forms a \emph{mock hyperbolic reflection space} if it admits the structure of a linear space such that three axioms are satisfied:
three points are collinear if and only if the product of their point-reflections is a point-reflection,
for any two points there is a unique midpoint, i.e. a unique point reflectiong one point to the other, and
given two distinct lines there is at most one point reflecting one line to the other.

We will consider in particular  mock hyperbolic reflection spaces arising from Frobenius groups of finite Morley rank. 
One of the main open problems about groups of finite Morley rank is the \emph{Algebraicity Conjecture}, which states that any infinite simple group of finite Morley rank should be an algebraic group over an algebraically closed field. While the conjecture was proved by Alt{\i}nel, Borovik, and Cherlin \cite{abc} in the characteristic $2$ setting, it is still wide open in general and in particular in the situation of small (Tits) rank. The conjecture would in fact imply that any sharply $2$-transitive group of finite Morley rank and, more generally, any Frobenius group of finite Morley rank splits. 

A Frobenius group is a group $G$ together with a proper nontrivial malnormal subgroup H, i.e. a subgroup $H$ such that $H \cap H^g = \{ 1 \}$ for all $g \in G \setminus H$. (Note that if $G$ is a \emph{bad group} of finite Morley rank  with Borel subgroup $B$ then $B< G$ is a Frobenius group.) A classical result due to Frobenius states that finite Frobenius groups split, i.e. they can be written as a semidirect product of a normal subgroup and the subgroup H. In the setting of finite groups the methods used by Frobenius play an important role in the classification of CA-groups, CN-groups, and groups of odd order.  For groups of finite Morley rank, all of the corresponding classification problems are still wide open.

Sharply $2$-transitive groups of finite Morley rank came to renewed attention since recently the first sharply $2$-transitive groups without non-trivial abelian normal subgroup were constructed in characteristic $2$ in \cite{rips-segev-tent} (see also \cite{TentZiegler}) and in characteristic $0$ in \cite{rips-tent}. However, as we show below, these groups do not have finite Morley rank. We also show that specific non-split sharply $2$-transitive groups of finite Morley rank would indeed be direct counterexamples to the Algebraicity Conjecture.

We prove the following splitting criteria for groups with an associated mock hyperbolic reflection space:

\begin{theorem}\label{thm:proj plane splits} If $G$ is a group with an associated mock hyperbolic reflection space $Q$, then the following are equivalent:
\begin{enumerate}
\item $G\cong A\rtimes \Cen(q)$ for some abelian normal subgroup $A$ and any $q \in Q$;
\item $Q$ is a (possibly degenerate) projective plane;
\item $Q$ consists of a single line.
\end{enumerate}\end{theorem}


We show a rank inequality for mock hyperbolic reflection spaces in groups of finite Morley rank: 
if $Q$ is a mock hyperbolic reflection space of Morley rank $n$ such that lines are infinite and of Morley rank $k$, then $n \leq 2k$ implies that $Q$ consists of a single line (and hence $n=k$). If $n = 2k+1$, then there exists a normal subgroup similar to the one in the above theorem (see \cref{thm:n=2k+1}). 

We then consider mock hyperbolic reflection spaces arising from Frobenius groups.
A connected Frobenius group $G$ of finite Morley rank with Frobenius complement $H$ falls into one of three classes: it is either degenerate, of odd or of even type depending on whether or not $G$ and $H$ contain involutions (see \cref{sec:frobenius_groups}). A connected Frobenius group is of odd type if and only if the Frobenius complement contains an involution. In particular, every sharply $2$-transitive groups of finite Morley rank and characteristic different from $2$ is a Frobenius group of odd type.
We show:

\begin{theorem}\label{cor:rank 6} Let $H<G$ be a connected Frobenius group of finite Morley rank and odd type.
\begin{enumerate}
\item The  involutions $J$ in $G$ form a mock hyperbolic reflection space and all lines are infinite. 
\item If a generic pair of involutions is contained in a line of Morley rank $k$ and  $\MR(J) \leq 2k+1$, then $H<G$ splits.
\item 
If $G$ does not split and  a generic pair of involutions is contained in a line of Morley rank $1$, then $G$ is a simple sharply $2$-transitive group of characteristic $\neq 2$ and hence a direct counterexample to the Algebraicity Conjecture.
\item If $\MR(G) \leq 10$, then either $G$ splits or $G$ is a simple non-split sharply $2$ transitive group of characteristic $\neq 2$ and $\MR(G)$ is either $8$ or $10$.
\end{enumerate}
\end{theorem}

For nilpotent Frobenius complements we show the following splitting criteria:

\begin{theorem}\label{thm:frobenius split or bad field} If $H<G$ is a connected Frobenius group of finite Morley rank and odd type with nilpotent complement $H$, then either of the following conditions implies that $H<G$ splits:
\begin{itemize}
\item $H$ is a minimal group, 
\item the lines in the associated mock hyperbolic reflection space have Morley rank $1$,
\item $G$ does not interpret a bad field of characteristic $0$.
\end{itemize}
\end{theorem}

If $G$ is a uniquely $2$-divisible Frobenius group, then $G$ does not contain involutions. However if the complement $H$ is abelian, then we can use a construction from the theory of K-loops to extend $G$ to a group containing involutions and if $H<G$ is full, i.e.,  if $G = \bigcup_{g \in G} H^g$, then the involutions in this extended group will again form a mock hyperbolic reflection space (see \cref{sec:frobenius_groups}). 

This construction allows us to use mock hyperbolic reflection spaces to study Frobenius groups of finite Morley rank and degenerate type. This class contains potential bad groups. Fr\'econ showed that bad groups of Morley rank $3$ do not exist \cite{frecon}. Subsequently Wagner~\cite{wagner} used Fr\'econ's methods to show more generally that if $H<G$ is a simple \emph{full} Frobenius group of Morley rank $n$ with abelian Frobenius complement $H$ of Morley rank $k$, then $n>2k+1$. Note that the existence of full Frobenius groups was claimed by Ivanov and Olshanski, but to the authors best knowledge no published proof exists (see also \cite{Jaligot} Fact 3.1).

If $H<G$ is a not necessarily full or simple Frobenius group of finite Morley rank and degenerate type, we obtain a weaker version of mock hyperbolic reflection spaces which  still allows us to extend Fr\'econ's and Wagner's results:

\begin{theorem}
If  $H<G$ is a connected Frobenius group of Morley rank $n$  and degenerate type with  abelian Frobenius complement  $H$ of Morley rank $k$, then  $n\geq 2k+1$.

If $n=2k+1$, then $G$ splits as $G = N \rtimes H$ for some definable connected normal subgroup~$N$.
Moreover, if $N$ is solvable, then there is an interpretable field $K$ of characteristic $\neq 2$ such that $G = K_+ \rtimes H$, $H \leq K^*$, and $H$ acts on $K_+$ by multiplication.
\end{theorem}

\section{Mock hyperbolic reflection spaces}

We now introduce the notion of mock hyperbolic reflection spaces, which will be central to our work.
The motivating example for our construction comes from sharply $2$-transitive groups in characteristic different from $2$ (see Section~\ref{sec:sharply_2-transitive}) in which the involutions have a rich geometric structure, which is reflected in the following definition.

Let $G$ be a group and $Q\subset G$ a conjugacy class  of involutions in $G$, and let $\Lambda\subset \calP(Q)$ be a $G$-invariant family of subsets of $Q$ such that each $\lambda \in \Lambda$ contains at least two elements.
We view involutions in $Q$ as points and elements of $\Lambda$ as lines so that the conjugation action of $Q$ on itself  corresponds to  point reflections.

For  involutions $i \neq j\in Q$ we write
\[ \ell_{ij} = \{ k \in Q : ij \in kQ \} \]
and we say that the line $\ell_{ij}$ \emph{exists in $\Lambda$} if $\ell_{ij}\in\Lambda$.

\begin{definition}\label{prop:partial_mock hyperbolic reflection} Let $G$ be a group, $Q\subset G$ a conjugacy class of involutions in $G$, and let $\Lambda\subset \calP(Q)$ be $G$ invariant such that each $\lambda \in \Lambda$ contains at least two elements.
The pair $(Q,\Lambda)$ is a \emph{partial mock hyperbolic reflection space} if the following conditions are satisfied:
\begin{enumerate}
\item For all  $\lambda \in \Lambda$ and $i\neq j\in\lambda$, we have
\[ \lambda = \ell_{ij}=\{ k \in Q : ij \in kQ \}. \]
 In particular, any two points are contained in at most one line.
\item Midpoints exist and are unique, i.e. given $i, j$ in $Q$ there is a unique $k \in Q$ such that $i^k = j$.
\item Given two distinct lines there is at most one point reflecting one line to the other. In other words, if $\lambda^i = \lambda^j$ for $i \neq j$ in $Q$, then $\lambda^i = \lambda = \lambda^j$.
\end{enumerate}
We say that  $(Q,\Lambda)$ is a  \emph{mock hyperbolic reflection space} if it satisfies a) -- c) and furthermore $\ell_{ij} \in\Lambda$ for all $i\neq j\in Q$.
\end{definition}

Given  a group $G$ and a conjugacy class $Q$ of involutions in $G$, in light of \cref{prop:partial_mock hyperbolic reflection} and in a slight abuse of notation, we  say that
$Q$ forms a {mock hyperbolic reflection space} if $(G, \{ \ell_{ij} : i \neq j \in Q \})$ is a mock hyperbolic reflection space. 

For a group $G$ and subset $A\subset G$ we write \[A^{\cdot n}=\{a_1\cdot\ldots\cdot a_n\mid  a_1,\ldots, a_n\in A\}\subseteq G.\]

\begin{remark}\label{rem:def} Let $(Q,\Lambda)$ be a partial mock hyperbolic reflection space.
We say that involutions $i, j, k\in Q$ are \emph{collinear} if there is  some $\lambda\in\Lambda$ with $i, j, k\in\lambda$.
\begin{itemize}
\item 
We will see below that if a) and b) hold, then c) is equivalent to either of the following conditions:
\begin{enumerate}
\item[c')] If $\lambda^i = \lambda^j$ for $i \neq j$ in $Q$ and $\lambda\in \Lambda$, then $i, j\in \lambda$.

\item[c'')] For every line $\lambda \in \Lambda$ we have $N_G(\lambda) \cap Q^{\cdot 2} = \lambda^{\cdot 2}$.
\end{enumerate}
\item 
If $(Q,\Lambda)$ is a mock hyperbolic reflection space, then a) is equivalent to:
\[ i, j, k\in Q \text{ are collinear if and only if } ijk\in Q. \]
\end{itemize}
\end{remark}

\begin{example}
Let $\bbH^n$ be the $n$-dimensional real hyperbolic space. 
Then $\Isom(\bbH^n)$, the group of all isometries of $\bbH^n$, 
contains the point-reflections as a conjugacy class $Q$ of involutions. $Q$ can be identified with $\bbH^n$ and hence $Q$ forms a mock hyperbolic reflection space. In case $n=2$ the simple group $\PSL_2(\bbR)$ consists of all orientation preserving isometries of $\bbH^2$. $\PSL_2(\bbR)$ is generated by the point-reflections and the point-reflections are the only involutions. 
\end{example}

\begin{example} \label{example:abelian} 
Let $A$ be a uniquely $2$-divisible abelian group and let $\epsilon \in \Aut(A)$ be given by $\epsilon(x) = x^\inv$. Put $G = A \rtimes \langle \epsilon \rangle$. Then the set of involutions in $G$ is given by $Q = A \times \{ \epsilon \}$ and $Q$ forms a mock hyperbolic reflection space consisting of a single line.
\end{example}

Other examples arise from sharply $2$-transitive groups (\cref{sec:sharply_2-transitive}) or can be constructed from a class of uniquely $2$-divisible Frobenius groups (\cref{sec:frobenius_groups}).

\begin{lemma}\label{lem:translations have no fixed point}
Let $G$ be a group and $Q$ a conjugacy class of involutions in $G$ such that $Q$ acts regularly on itself by conjugation, i.e. $Q$ satisfies condition b) in \cref{prop:partial_mock hyperbolic reflection}. Then the following holds for any $i \in Q$:
\begin{enumerate}
\item $iQ$ is uniquely $2$-divisible.
\item $Q^{\cdot 2} \cap \Cen(i) = \{1\}$.
\item \label{prop:decomposition}  $G = (iQ) \Cen(i)$ and every $g\in G$ can be written uniquely as $g=ijh$ with $j\in Q$ and $h\in\Cen(i)$.
\end{enumerate}
\end{lemma}
\begin{proof}
a) Fix $ia \in iQ$. We have to show that there is a unique $b \in Q$ such that
\[ ia = (ib)^2 = ibib = ii^b. \]
This is exactly condition b) in \cref{prop:partial_mock hyperbolic reflection}.

b) Suppose $a$ and $b$ are involutions in $Q$ such that $i^{ab}=i$. Then $i^a = i^b$ and hence $a = b$ by the uniqueness in condition b). Hence $ab=1$.

c) Let $g \in G$ and set $k = i^{g^\inv}$. Then there is a unique $j \in Q$ such that $k^{ij} = i$. Now put $h = jig$. Then $g = ijh$ and we have
\[ k^{ij} = i = k^g = k^{ijh} = i^h \]
and therefore $h \in \Cen(i)$.
This shows  existence of such a decomposition, uniqueness follows from part b).
\end{proof}

In accordance with the terminology from real hyperbolic spaces or from sharply $2$-transitive groups we call elements of the set 
\[S = \{ \sigma \in Q^{\cdot 2} \setminus \{ 1 \} : \ell_\sigma \text{ exists in } \Lambda \} \cup \{ 1 \} \]
\emph{translations}. Then Part b) of \cref{lem:basic_properties} implies that translations have no fixed points (in their action on $Q$).

B. H. Neumann showed that a uniquely $2$-divisible group admitting a fixed point free involutionary automorphism must be abelian \cite{neumann}. More generally, uniquely $2$-divisible groups with involutionary automorphisms can be decomposed as follows:

\begin{proposition}[Exercise 14 on p. 73 of \cite{borovik-nesin}] \label{prop:involutionary_decomposition} \label{prop:neumann_abelian}
Let $G$ be a uniquely $2$-divisible group and let $\alpha \in \Aut(G)$ be an involutionary automorphism.
Define the sets $\Inv(\alpha) = \{ g \in G : g^\alpha = g^\inv \}$ and $\Cen(\alpha)=\{ g \in G : g^\alpha = g \}$. 

Then $G = \Inv(\alpha) \Cen(\alpha)$, and for every $g \in G$ there are unique $a \in \Inv(\alpha), b \in \Cen(\alpha)$ such that $g = ab$. In particular, if $\alpha$ has no fixed points, then $G$ is abelian and $\alpha$ acts by inversion.
\end{proposition}

\begin{lemma} \label{lem:basic_properties}
Suppose $(Q,\Lambda)$ satisfies conditions a) and b) in \cref{prop:partial_mock hyperbolic reflection}. Let $\lambda$ be a line in $\Lambda$. 
\begin{enumerate}
\item $N_G(\lambda) \cap Q = \lambda$.
\item If $i \in \lambda$, then $\lambda^{\cdot 2} = i\lambda$.
\item If $a,b$ are distinct involutions in $Q$ such that $ab \in \lambda^{\cdot 2}$, then $a,b \in \lambda$.
\item $\lambda^{\cdot 2}$ a uniquely $2$-divisible abelian group.
\item If $i,j, k\in Q$ are such that $\ell_{ij}$ and $\ell_{jk}$ exist in $\Lambda$, then 
$\ell_{ij}^{\cdot 2}\cdot \ell_{jk}^{\cdot 2}=\ell_{ij}\cdot \ell_{jk}\subseteq Q^{\cdot 2}$.
\item $N_G(\lambda) = N_G(\lambda^\squared)$.
\end{enumerate}
\end{lemma}
\begin{proof}
a) We first show $\lambda \subseteq N_G(\lambda)$: If $\lambda = \ell_{ij}$, then $j^i = iji$ and hence $ij \in j^iQ$. Therefore $j^i \in \lambda$.

Now assume $k \in N_G(\lambda) \cap Q$ and $\lambda = \ell_{ij}$.
We may assume $k \neq i$. Then $i \neq i^k \in \lambda$ and hence $\lambda = \ell_{i^ki}$.
Now $i^ki = kk^i$ hence $i^ki \in kQ$ and therefore $k \in \lambda$.

b) Fix $a \neq b$ in $\lambda$. Then $ab \in iQ$ and hence $ab = ij$ for some $j \in Q$. It remains to show $j \in \lambda$: We have $ab = ij \in Qj = jQ$ and hence $j \in \lambda$.

c) Suppose $ab = ij$ and $\lambda = \ell_{ij}$. Then $(ij)^a = (ij)^\inv = ji$ and therefore $\lambda^a = \lambda$ and hence $a \in N_G(\lambda) \cap Q = \lambda$.  Now $a\lambda = \lambda^{\cdot 2}$ and hence $b \in \lambda$.

d) We first show that $\lambda^{\cdot 2}=i\lambda$ is uniquely $2$-divisible. Since we know that  $iQ$ is uniquely $2$-divisible, it remains to show that $i\lambda$ is $2$-divisible. Fix $ia \in i\lambda$, say $ia = (ib)^2$ for some $b \in Q$.
Then $ia = ii^b$, hence $a = i^b$ and therefore $b \in N_G(\lambda) \cap Q = \lambda$.

It remains to show that $\lambda^{\cdot 2} = i\lambda$ is an abelian group. Note that $i\lambda = \lambda i$ and hence $\lambda^{\cdot 2}$ is closed under multiplication and taking inverses. Therefore $\lambda^{\cdot 2}$ is a uniquely $2$-divisible group. Moreover, $i$ acts on $\lambda^{\cdot 2}$ as an involutionary automorphism without fixed points. Now \cref{prop:neumann_abelian} implies that $\lambda^{\cdot 2}$ is abelian. 

e) Since $j$ normalizes $\ell_{ij}$, we have $\ell_{ij}^{\cdot 2}=\ell_{ij}j$ by b), and hence the claim follows.

f) The inclusion $N_G(\lambda) \subseteq N_G(\lambda^\squared)$ is obvious. Therefore we only need to show $N_G(\lambda^\squared) \subseteq N_G(\lambda)$. Take $g \in N_G( \lambda^\squared ) \setminus \{ 1 \}$ and fix $i \neq j \in \lambda$. Then $ij \in \lambda^\squared$ and hence $i^g j^g \in \lambda^\squared$. Therefore $i^g, j^g \in \lambda$ by c) and thus $\lambda^g = \lambda$.
\end{proof}

\begin{lemma} \label{lem:mock hyperbolic reflection_equivalences}
Suppose $(Q,\Lambda)$ satisfies a) and b) in \cref{prop:partial_mock hyperbolic reflection}. Then the following are equivalent:
\begin{enumerate}
\item $(Q,\Lambda)$ is a partial mock hyperbolic reflection space.
\item Every line $\lambda \in \Lambda$ satisfies $N_G(\lambda) \cap Q^{\cdot 2} = \lambda^{\cdot 2}$.
\end{enumerate}
\end{lemma}
\begin{proof}
Suppose $(Q,\Lambda)$ forms a partial mock hyperbolic reflection space and fix $ij \in N_G(\lambda) \cap Q^{\cdot 2}$ and assume $i \neq j \in Q$. Then $\lambda^i = \lambda = \lambda^j$ and hence $i,j \in N_G(\lambda) \cap Q = \lambda$. Therefore $N_G(\lambda) \cap Q^{\cdot 2} = \lambda^{\cdot 2}$.

Conversely, assume $N_G(\lambda) \cap Q^{\cdot 2} = \lambda^{\cdot 2}$ and $\lambda^i = \lambda^j$ for $i \neq j \in Q$. Then $ij \in \lambda^{\cdot 2}$ and hence $i,j \in \lambda$ by \cref{lem:basic_properties} c). This shows $\lambda^i = \lambda = \lambda^j$.
\end{proof}

\begin{proposition}\label{prop:partition}
Let $G$ be a group, $Q$ a conjugacy class of involutions in $G$ and
suppose $(Q,\Lambda)$ is a partial mock hyperbolic reflection space. 
Then the following holds:
\begin{enumerate}
\item If $\lambda = \ell_{ij}\in\Lambda$, then $\lambda^\squared = iQ \cap jQ = \Cen(ij) \cap Q^{\cdot 2}$.
\item The set $S\setminus\{1\} = \{ ij\in Q^{\cdot 2} \setminus \{ 1 \} : \ell_{ij} \text{ exists in } \Lambda \}$ is partitioned by the family $\{ \lambda^{\cdot 2} \setminus \{ 1 \} : \lambda \in \Lambda \}$.
\end{enumerate}

\end{proposition}
\begin{proof}
Clearly, 2. follows from 1. In order to prove 1., we first show $\lambda^{\cdot 2} = iQ \cap jQ$. Fix $ia = jb \in iQ \cap jQ$. Then $ab = ij \in \lambda^{\cdot 2}$ and hence $a,b \in \lambda$ by \cref{lem:basic_properties} c). This shows $iQ \cap jQ \subseteq \lambda^{\cdot 2}$. Moreover, we have $\lambda^{\cdot 2} = i\lambda = j\lambda$ and hence $\lambda^{\cdot 2} \subseteq iQ \cap jQ$. Thus $\lambda^{\cdot 2} = iQ \cap jQ$.

$\lambda^{\cdot 2}$ is an abelian group and contains $ij$. Therefore $\lambda^{\cdot 2} \subseteq \Cen(ij) \cap Q^{\cdot 2}$. An element $g \in \Cen(ij)$ normalizes $\lambda = \ell_{ij} = \ell_{ji}$ and hence normalizes $\lambda^{\cdot 2}$. Therefore $\Cen(ij) \cap Q^{\cdot 2} \subseteq N_G(\lambda^{\cdot 2}) \cap Q^{\cdot 2} = \lambda^{\cdot 2}$ and hence $\Cen(ij) \cap Q^{\cdot 2} = \lambda^{\cdot 2}$.

\end{proof}

If $i$ is an involution in $Q$, then we set $\Lambda_i = \{ \lambda \in \Lambda : i \in \lambda \}$ to be the set of all lines that contain $i$.
\begin{proposition} \label{line_lemma}
Suppose $(Q,\Lambda)$ forms a partial mock hyperbolic reflection space.
\begin{enumerate}
\item Suppose $\lambda \cap \lambda^j \neq \emptyset$ for a line $\lambda$ and an involution $j$ in $Q$. Then $j \in \lambda$ and therefore $\lambda = \lambda^j$.
\item $G$ acts transitiviely on $\Lambda$ if and only if $\Cen(i)$ acts transitively on $\Lambda_i$ for each $i \in Q$.
\end{enumerate}
\end{proposition}
\begin{proof}
a) Suppose $\{i\} = \lambda \cap \lambda^j$. Then $i = i^j$and therefore $j = i \in \lambda$.

b) If $\Cen(i)$ acts transitively on $\Lambda_i$, then $G$ is transitive on $\Lambda$ because all involutions in $Q$ are conjugate.

Now assume $G$ acts transitively on $\Lambda$ and suppose $i \in \lambda \cap \lambda^g$ for some $g \in G$. By \cref{prop:decomposition} $g$ can be written as $g = ijh$ for some $j \in Q$ and $h \in \Cen(i)$.
Note that $\lambda^g = \lambda^{ijh} = \lambda^{jh}$ because $i$ is contained in $\lambda$.

Since $h \in \Cen(i)$ this implies that $i$ must be contained in $\lambda^j$ and hence \[ i \in \lambda \cap \lambda^j. \] Therefore a) implies that $j$ must be contained in $\lambda$ and hence $\lambda = \lambda^j$. Hence $\lambda^g = \lambda^h$. Since $g$ was arbitrary, this shows that $\Cen(i)$ acts transitively on $\Lambda_i$.
\end{proof}

\subsection{The geometry of a mock hyperbolic reflection space}

Recall that a mock hyperbolic reflection space is a partial hyperbolic space such that any two points are contained in a line. 
As a first step we show that the geometry of a mock hyperbolic reflection space cannot contain a proper projective plane:

\begin{lemma}\label{lem:intersections}
Suppose that $(Q,\Lambda)$ is a mock hyperbolic reflection space in a group $G$ and $X\subseteq Q$ is a projective plane, i.e. \begin{enumerate}
    \item for all $\ i \neq j \in X$ the line $\ell_{ij}$ is contained in $X$, and
    \item if $\lambda$ and $\delta$ are lines contained in $X$ then $\lambda \cap \delta \neq \emptyset$,
  \end{enumerate}
 Then $X^{\cdot 2}$ is a subgroup of $G$.
\end{lemma}
\begin{proof}
	This follows at once from \cref{lem:basic_properties} e).
\end{proof}

\begin{lemma}
Suppose $Q$ forms a mock hyperbolic reflection space in a group $G$ and
let $H \subseteq Q^{\cdot 2}$ be a subgroup of $G$ which is uniquely 2-divisible and normalized by an involution $i \in Q$. Then $H \subseteq \Cen(\sigma)$ for some $\sigma \in Q^{\cdot 2} \setminus \{ 1 \}$.
\end{lemma}
\begin{proof}
	Since $H$ is uniquely 2-divisible and $i$ acts as an involutionary automorphism without fixed points, \cref{prop:neumann_abelian} implies that $H$ is abelian and hence must be contained in the centralizer of some translation.
\end{proof}

\begin{proposition}\label{prop:no proj plane} If $Q$ is a mock hyperbolic reflection space in a group $G$, then it does not contain a proper projective plane. I.e. if $X \subseteq Q$ is a projective plane, 
  then $X$ contains at most one line.
\end{proposition}
\begin{proof}
 By \cref{lem:intersections} the set $X^{\cdot 2}$ is  a subgroup of $G$. Moreover, $X^{\cdot 2}$ is uniquely 2-divisible since it is a union of centralizers of translations.
  Each $j \in X$ acts on $X^{\cdot 2}$ as an involutionary automorphism without fixed points. By the previous lemma $X^{\cdot 2} \leq \Cen(\sigma)$ and hence $X \subseteq \ell_\sigma$.
\end{proof}

\begin{theorem}\label{prop:BHNeumann}
Suppose $Q$ forms a mock hyperbolic reflection space in a group $G$. Then the following are equivalent:
\begin{enumerate}
\item $\Lambda$ consists of a single line;
\item $Q$ is a projective plane;
\item $G$ has an abelian normal subgroup $A \not \subseteq \bigcap_{i \in Q} \Cen(i)$;
\item $Q^{\cdot 2} = iQ$ for any involution $i \in Q$;
\item $iQ$ is commutative for any involution $i \in Q$;
\item $iQ$ is a subgroup of $G$ for any involution $i \in Q$;
\item $Q^{\cdot 2}$ is a subgroup of $G$;
\item $iQ$ is an abelian normal subgroup of $G$ and $G$ splits as $G = iQ \rtimes \Cen(i)$ for any involution $i \in Q$.
\end{enumerate}
\end{theorem}
\begin{proof}
We show the following implications:
\[ \text{d)} \iff \text{a)} \implies \text{b)} \implies \text{g)} \implies \text{e)} \implies \text{f)} \implies \text{h)} \implies \text{c)} \implies \text{a)}. \]
To show a) $\iff$ d):
 Assume d) and fix a line $\lambda = \ell_{ij}$. Then 
\[ i\lambda = \lambda^{\cdot 2} = iQ \cap jQ = Q^{\cdot 2} = iQ \]
and hence $\lambda = Q$ is the only line.
Conversely, assume a) holds and $\lambda = Q$ is the unique line, then $Q^{\cdot 2} = \lambda^{\cdot 2} = i\lambda = iQ$ by \cref{lem:basic_properties}.

a) $\implies$ b) is trivial.

b) $\implies$ g) holds by \cref{lem:intersections}.

Now assume g) holds.  By \cref{prop:partition} $Q^\squared \setminus \{ 1 \}$ is partitioned by the family $\{ \lambda^\squared \setminus \{ 1 \} : \lambda \in \Lambda \}$. Each $\lambda^\squared$ is a uniquely $2$-divisible abelian group by \cref{lem:basic_properties}. Therefore $Q^\squared$ is uniquely $2$-divisible. If $i$ is any involution, then $i$ normalizes $Q^\squared$ and acts by conjugation as an involutionary automorphism without fixed points. Therefore $Q^\squared$ is an abelian group by \cref{prop:neumann_abelian}. In particular, $iQ \subseteq Q^\squared$ is commutative. This shows e).

Now assume e). $iQ$ is partitioned by $\{ \lambda^\squared : \lambda \in \Lambda, i \in \lambda \}$ and if $\lambda = \ell_{ij}$, then $\lambda^\squared = \Cen(ij) \cap Q^\squared$. Since $iQ$ is commutative this implies $iQ = \lambda^\squared$ for some line $\lambda$ and hence $iQ$ is a subgroup of $G$ by \cref{lem:basic_properties}. This shows f).

We next show f) $\implies$ h): $iQ$ is a uniquely $2$-divisible group and $i$ acts as an involutionary automorphism without fixed point. Therefore $iQ$ is an abelian subgroup of $G$ by \cref{prop:neumann_abelian}. Note that $N_G(iQ)$ contains $\Cen(i)$ and $iQ$. Therefore $G = iQ\Cen(i) = N_G(iQ)$ by \cref{prop:decomposition}. Hence $iQ$ is an abelian normal subgroup of $G$ and therefore $G = iQ \rtimes \Cen(i)$ by \cref{prop:decomposition}.

h) $\implies$ c) is obvious.

To see that c) implies a), let  $i \in Q$ and  $a \in A \setminus \Cen(i)$. Then
\[ 1 \neq a^\inv a^i = i^a i \in A \cap iQ. \]
In particular, $A \cap iQ$ is nontrivial. Now fix $\sigma \in (A \cap iQ) \setminus \{ 1 \}$ and set
$ \lambda = \ell_\sigma$. Then
\[ \Cen(\sigma) \cap Q^{\cdot 2} = \lambda^{\cdot 2} \]
and hence $A \cap Q^{\cdot 2} \subseteq \lambda^{\cdot 2}$. This implies that $\lambda^\squared$ is a normal subset of $G$ and hence $\lambda$ is a normal subset of $G$ by \cref{lem:basic_properties}. Therefore $\lambda = Q$ by \cref{lem:basic_properties}.
\end{proof}

\section{Sharply 2-transitive groups}\label{sec:sharply_2-transitive}

A permutation group $G$ acting on a set $X$, $|X| \geq 2$, is called \emph{sharply 2-transitive} if it acts regularly on pairs of distinct points, or, equivalently, if $G$ acts transitively on $X$ and for each $x\in X$ the point stabilizer $G_x$ acts regularly on $X \setminus \{x\}$. For two distinct elements $x,y \in X$ the unique $g \in G$ such that $(x,y)^g = (y,x)$ is an involution. Hence the set $J$ of involutions in $G$ is non-empty and forms a conjugacy class. 

The (permutation) characteristic of a group $G$ acting  sharply 2-transitively on a set $X$ is defined as follows: put $\chara(G) = 2$ if and only if involutions have no fixed points. If involutions have a (necessarily unique) fixed point,  the $G$-equivariant bijection $i \mapsto \fix(i)$ allows us to identify the given action of  $G$ on $X$ with the conjugation action of $G$ on $J$. Thus in this case, the set $S\setminus\{1\}$ of nontrivial translations also forms a single conjugacy class.  We put $\chara(G)=p$ (or $0$) if  translations have order $p$ (or infinite order, respectively).
For the standard examples of sharply 2-transitive groups, namely $K\rtimes K^*$ for some field $K$, this definition of characteristic agrees with the characteristic of the field $K$.

\begin{remark}\label{lem:basic} Let $G$ be a sharply 2-transitive group of characteristic $\chara(G) \neq 2$. Since $G$ acts sharply $2$-transitively by conjugation on the set $J$ of involutions in $G$, the following properties are easy to see:
\begin{enumerate}
  \item $\Cen(i)$ acts regularly on $J \setminus \{i\}$,
  \item the set $J$ acts regularly on itself by conjugation,  i.e. Condition b) of \cref{prop:partial_mock hyperbolic reflection} holds.
  \item $J^{\cdot 2} \cap \Cen(i) = \{1\}$ for all $i \in J$.
\end{enumerate}
In particular, a nontrivial translation does not have a fixed point.
\end{remark}

In order to define the lines for a mock hyperbolic reflection space on $J$, we need the following equivalent conditions to be satisfied:

\begin{proposition}\label{lem:geometry} If $\chara(G)\neq 2$, the following conditions are  equivalent:
\begin{enumerate}
\item Commuting is transitive on $J^{\cdot 2} \setminus \{1\}$.
\item $iJ \cap kJ$ is uniquely 2-divisible for all involutions $i\neq k\in J$.
\item $\Cen(ik) = iJ \cap kJ$ is abelian and is inverted by $k$ for all $i\neq k\in J$. 
\item The set $\{\Cen(\sigma) \setminus \{1\} : \sigma \in J^{\cdot 2} \setminus \{ 1 \} \}$ forms
a partition of $J^{\cdot 2} \setminus \{1\}$.
\end{enumerate}  
 
Furthermore, these conditions are satisfied in split sharply 2-transitive groups, whenever $\chara(G)=p\neq 0, 2$ or in case $G$ satisfies the descending chain condition for centralizers, so in particular if  $G$ has finite Morley rank. 
\end{proposition}
\begin{proof}
(a) $\Rightarrow$ (b):   Note that  since $(ij)^2=ii^j\in iJ$ every element of $iJ$ has a unique square-root in $iJ$.  Let $\tau \in iJ \cap kJ$. By assumption the group $A = \langle \Cen(\tau) \cap J^{\cdot 2} \rangle \leq \Cen(\tau)$ is abelian. Moreover, $A \cap J = \emptyset$ by Remark~\ref{lem:basic}. Hence the square-map is an injective group homomorphism from $A$ to $A$.
  
  There is $\sigma_i \in iJ$ such that $\sigma_i^2 = \tau$ and therefore $\sigma_i \in \Cen(\tau) \cap iJ$ because commuting is transitive. Similarly we  find $\sigma_k \in \Cen(\tau) \cap kJ$ such that $\sigma_k^2 = \tau$. Since the square-map is injective, it follows that $\sigma_i=\sigma_k \in iJ \cap kJ$. Therefore $iJ \cap kJ$ is uniquely 2-divisible. 

(b) $\Rightarrow$ (c) is contained in  Lemma 11.50 iv of \cite{borovik-nesin}.

(c) $\Rightarrow$ (d)  and  (d) $\Rightarrow$ (a) are obvious.

These conditions hold in split sharply 2-transitive groups by Theorem~\ref{thm:neumann}.  Furthermore, if $\chara(G)=p\neq 0, 2$ or if $G$ satisfies the descending chain condition for centralizers, these conditions hold by Lemma 11.50 of \cite{borovik-nesin}.
\end{proof}
 
The examples constructed in \cite{rips-segev-tent} (see also \cite{TentZiegler}) show that in characteristic 2 these conditions need not be satisfied. The non-split examples in characteristic 0 constructed in \cite{rips-tent} satisfy the assumptions and it is an open question whether non-split sharply 2-transitive groups exist in characteristic 0 which fail to satisfy these conditions. Note that Lemma 2.3 and 5.3 of \cite{rips-tent} imply that the maximal near-field in these examples is planar.

Assume now that the conditions of \cref{lem:geometry} are satisfied. Then for $i\neq j\in J$ we put
\[ \ell_{ij} = \{ k \in J : ij \in kJ \} \mbox{ and }
 \Lambda=\{\ell_{ij}\colon i\neq j\in J\}.\]
By \cref{lem:geometry}, $(J,\Lambda)$ satisfies Conditions a) and b) of \cref{prop:partial_mock hyperbolic reflection} and we  have
 \[ \ell_{ij} = \{ k \in J : ij \in kJ \} = i\ \Cen(ij)=\{k\in J: (ij)^k=ji\}. \]
 
 The point-line geometry $(J,\Lambda)$ is equivalent to the incidence geometry considered by Borovik and Nesin in Section 11.4 of \cite{borovik-nesin}.
 
 If $\lambda = \ell_{ij}$ is a line, then $N_G(\lambda) = N_G(\Cen(ij))$ is a split sharply $2$-transitive group
 \[ N_G(\lambda) = \Cen(ij) \rtimes N_{\Cen(i)}(\lambda) \]
 and corresponds to the maximal near-field. If the maximal near-field is planar, then
 \[ N_G(\lambda) = \Cen(ij) \cup \bigcup_{k \in \lambda} N_{\Cen(k)}(\lambda). \]

\begin{lemma} \label{line_lemma2} Assume that  $G$ is sharply 2-transitive, $\chara(G)\neq 2$ and  the conditions of \cref{lem:geometry} are satisfied. Assume moreover that the maximal near-field is planar. If $\lambda\in\Lambda$ and $i\neq j\in J$ such that $\lambda^i = \lambda^j$, then $i,j \in \lambda$ and so $\lambda^i = \lambda = \lambda^j$, and so Condition c) of \cref{prop:partial_mock hyperbolic reflection} holds.
\end{lemma}
\begin{proof}
This is contained in the proof of Theorem 11.71 in \cite{borovik-nesin}.
Since our definition of lines is slightly different from the one given in \cite{borovik-nesin}, we include a proof. If $\lambda^i = \lambda^j$ then $ij \in N_G(\lambda)$ and hence $ij \in N_G(\lambda^{\cdot 2})$. Now $\lambda^{\cdot 2} = \Cen(\sigma)$ for some $\sigma \in J^{\cdot 2} \setminus \{ 1 \}$ such that $\lambda = 
	\ell_\sigma$. Fix $s \in \lambda$. The group $N_G(\Cen(\sigma)) = \Cen(\sigma) \rtimes N_{\Cen(s)}(\Cen(\sigma))$ is split sharply 2-transitive by Proposition 11.51 of \cite{borovik-nesin}.
	Since the maximal neqr-field is planar we have
	\[ ij \in N_G(\Cen(\sigma)) \cap J^{\cdot 2} = \Cen(\sigma) \]
	and therefore $i,j \in \ell_\sigma = \lambda$.
\end{proof}

\begin{corollary}\label{cor:mock hyperbolic reflection from 2-sharp}
Let $G$  be a sharply 2-transitive group. Then the set of involutions $J\subset G$ forms a mock hyperbolic reflection space in either of the following cases:
\begin{enumerate}
\item $G$ is a split sharply $2$-transitive group corresponding to a planar near-field of characteristic $\neq 2$;
\item  $\chara(G)=p\neq 0, 2$ and the maximal near-field is planar; or
\item $\chara(G) = 0$, $G$ satisfies the descending chain condition for centralizers, and the maximal near-field is planar.
\end{enumerate}
In particular, if $\chara(G) \neq 2$ and $G$ is of finite Morley rank, then the involutions in $G$ form a mock hyperbolic reflection space.
\end{corollary}

In the case of sharply 2-transitive groups, \cref{prop:no proj plane} reduces to the following well-known result of B. H. Neumann~\cite{neumann}:

\begin{theorem}\label{thm:neumann}
  A sharply 2-transitive group $G$ splits if and only if the set of translations $J^{\cdot 2}$ is a subgroup of $G$ (and in that case, $J^{\cdot 2}$ must in fact be abelian).
\end{theorem}

\section{Uniquely 2-divisible Frobenius groups} \label{sec:frobenius_groups}
In this section we will construct (partial) mock hyperbolic reflection spaces from uniquely $2$-divisible Frobenius groups with abelian Frobenius complement. This construction makes use of K-loops and quasidirect products.

\subsection{K-Loops and quasidirect products}
K-loops are non-associative generalizations of abelian groups. They are also known as Bruck loops and gyrocommutative gyrogroups.
We mostly follow Kiechle's book \cite{kiechle}.

\begin{definition}
A groupoid $(L, \cdot, 1)$ is a \emph{K-loop} if
\begin{enumerate}
\item it is a loop, i.e. the equations
\[ ax = b \quad \text{and} \quad xa = b \]
have unique solutions for all $a,b \in L$,
\item it satisfies the Bol condition, i.e.
\[ a(b \cdot ac ) = (a \cdot ba)c \]
for all $a,b,c \in L$, and
\item is satisfies the automorphic inverse property, i.e. all elements of $L$ have inverses and we have
\[(ab)^\inv = a^\inv b^\inv \]
for all $a,b \in L$.
\end{enumerate}
\end{definition}
Given $a \in L$ let $\lambda_a : L \rightarrow L$ be defined by $\lambda_a(x) = ax$. Given $a,b \in L$ we define the precession map
\[ \delta_{a,b} = \lambda_{ab}^\inv \lambda_a \lambda_b. \]
These maps are characterized by
\[ a\cdot bx = ab \cdot \delta_{a,b}(x) \quad \text{for all } x \in L. \]
If $L$ is a K-loop, then the precession maps are automorphisms and we set
\[ \calD = \calD(L) = \langle \delta_{a,b} : a,b \in L \rangle \leq \Aut(L). \]
 
 The following identities will be used in this section:
 \begin{proposition}
 Let $L$ be a K-loop, $a,b \in L$, and $\alpha \in \Aut(L)$. Then the following identities hold:
 \begin{enumerate}
 \item $\alpha^\inv \delta_{a,b} \alpha = \delta_{ \alpha^\inv (a), \alpha^\inv (b)}$,
 \item $\delta_{a, a^\inv} = \id$,
 \item $\delta_{a, ba} = \delta_{a,b}$,
 \item $\delta_{a,b} = \delta_{b,a}^\inv$,
 \item $\delta_{a,b} = \delta_{a^\inv, b^\inv}$.
 \end{enumerate}
 \end{proposition}
 \begin{proof}
 The first identity follows from \cite[2.4]{kiechle}. The second identity is a consequence of \cite[6.1]{kiechle}. Theorem 6.4 of \cite{kiechle} shows the third identity. The remaining two identities are part of \cite[Theorem 3.7]{kiechle}.
 \end{proof}
 
 \begin{definition} \label{def:twisted_subgroup}
 Let $G$ be a group. A subset $L \subseteq G$ is a twisted subgroup of $G$ if and only if $1 \in L$, $L^\inv \subseteq L$, and $aLa \subseteq L$ for all $a \in L$.
 \end{definition}
 Note that twisted subgroups are closed under the square-map. A twisted subgroup is \emph{uniquely $2$-divisible} if the square map is bijective.
 
\begin{proposition}[Theorem 6.14 of \cite{kiechle}] \label{prop:twisted_subgroup}
Let $G$ be a group with a uniquely $2$-divisible twisted subgroup $L \subseteq G$. Then
 \[ a \otimes b = a^\half b a^\half \]
 makes L into a K-loop $(L,\otimes, 1)$ and integer powers of elements in $L$ agree in $G$ and $(L,\otimes)$. Moreover, given $a,b \in L$ the precession map $\delta_{a,b}$ is given by conjugation with
 \[ d_{a,b} = b^\half a^\half (a^\half b a^\half)^\ihalf. \]
 \end{proposition}
 
 \begin{proposition}[Theorem 2.13 of \cite{kiechle}] \label{prop:quasidirect_product}
 Let $L$ be a K-loop and let $\calA \leq \Aut(L)$ a group of automorphisms such that $\calD(L) \subseteq \calA$. Then
 \begin{enumerate}
 \item The quasidirect product $L \rtimes_Q \calA$ given by the set $L \times \calA$ together with the multiplication
 \[ (a,\alpha)(b,\beta) = (a \cdot \alpha(b), \delta_{a, \alpha(b)} \alpha \beta) \]
 forms a group with neutral element $(1, \id)$. Inverses are given by
 \[ (a,\alpha)^\inv = ( \alpha^\inv(a^\inv), \alpha^\inv ). \]
 \item $L \rtimes_Q \calA$ acts faithfully and transitively on $L$ by
 \[ (a,\alpha)(x) = a \alpha(x) \quad \text{for all } (a,\alpha) \in L \rtimes_Q \calA, x \in L. \]
 \end{enumerate}
 \end{proposition}
  
 \subsection{Mock hyperbolic reflection spaces from uniquely 2-divisible Frobenius groups}
Let $H<G$ be a uniquely 2-divisible Frobenius group with abelian complement $H$. 

We set $L$ to be the K-loop $L = (G, \otimes)$ where $\otimes$ is defined by
\[ a \otimes b = a^\half b a^\half. \]
Set $\calA = G \times \langle \epsilon \rangle < \Aut(L)$ where $\epsilon$ inverts all elements of $L$. Put $\calG = L \rtimes_Q \calA$. Let $J$ be the set of all involutions in $\calG$ and put $\iota = (1,\epsilon) \in J$.

\begin{lemma} \label{lem:quasidirect_involutions}
\begin{enumerate}
\item $J = L \times \{ \epsilon \}$.
\item $\Cen(\iota) = 1 \times \calA$.
\item for all $i,j \in J$ there is a unique $k \in J$ such that $j = i^k$.
\end{enumerate}
\end{lemma}
\begin{proof}
a) Fix $(a, \alpha) \in \calG$ such that $(a,\alpha)^2 = (1,\id)$. Note that
\[ (a,\alpha)(a, \alpha) = (a \otimes \alpha(a), \delta_{a, \alpha(a)} \alpha^2 ). \]
Now $a \otimes \alpha(a) = 1$ implies $\alpha(a) = a^\inv$ and therefore $\delta_{a,\alpha(a)} = \id$.
Hence we must have $\alpha^2 = \id$.

If $\alpha = \id$, then $a\otimes \alpha(a) = a^2$, therefore $a^2 = 1$ and thus $a = 1$. In that case $(a,\alpha) = (1,\id)$ is the neutral element in $\calG$.

This shows $J = L \times \{ \epsilon \}$ because $\epsilon$ is the only involution in $\calA$.

b) Fix $(a,\alpha) \in \Cen( \iota )$. We have
\begin{align*}
(a, \alpha) (1, \epsilon) &= (a, \alpha \epsilon) \quad \text{and} \\
(1, \epsilon)(a, \alpha) &= (a^\inv , \epsilon \alpha ).
\end{align*}
Hence $(a,\alpha) \in \Cen( \iota)$ if and only if $a = a^\inv$ if and only if $a = 1$.

c) Take involutions $(a,\epsilon),(b,\epsilon),(c,\epsilon) \in J = L \times \{ \epsilon \}$. Then
\begin{align*}
(b,\epsilon)(a,\epsilon)(b,\epsilon) &= (b,\epsilon)(a \otimes b^\inv, \delta_{a,b^\inv}) \\
&= (b \otimes (a^\inv \otimes b), \delta_{b, a^\inv \otimes b } \delta_{a, b^\inv} \epsilon ) \\
&= ((b \otimes a^\ihalf)^2, \epsilon ).
\end{align*}
Hence we have $(a,\epsilon)^{(b,\epsilon)} = (c, \epsilon)$ if and only if $b \otimes a^\ihalf = c^\half$. The loop conditions ensure that for all $a,c \in L$ there is a unique $b$ satisfying this equation.
\end{proof}

Now set $\lambda_0 = H \times \{\epsilon \} \subseteq J$ and put $\Lambda = \{ \lambda_0^g : g \in \calG \}$.
We view elements of $\Lambda$ as lines and we view involutions as points. Note that $\Lambda$ is $\calG$-invariant and all lines are conjugate.

The following will be shown in this section:
\begin{theorem} \label{thm:frobenius_mock hyperbolic reflection}
\begin{enumerate}
\item $(J,\Lambda)$ is a partial mock hyperbolic reflection space in $\calG$. 
\item If $G$ is full, i.e., if $G = \bigcup_{g \in G}H^g$,
then $(J,\Lambda)$ is a mock hyperbolic reflection space.
\item Suppose $i,j,k \in J$ are pairwise distinct such that the lines $\ell_{ij}$ and $\ell_{ik}$ exist and assume that $i,j,k$ are not collinear. Then $\Cen_\calG(i,j,k) = 1$. In particular, $\calG$ acts faithfully on $J$. 
\end{enumerate}
\end{theorem}

\begin{lemma}
  Let $\lambda$ be a line containing $\iota$. Then $\lambda$ is of the form
  \[ \lambda = H^g \times \{ \epsilon \} \]
  for some $g \in G$.
\end{lemma}
\begin{proof}
We have $\lambda = \lambda_0^g$ for some $g = (a,\alpha) \in \calG$. Hence elements of $\lambda$ are of the form
\begin{align*}
 & (\alpha^\inv(a^\inv), \alpha^\inv)(c,\epsilon)(a,\alpha) \\
 = & (\alpha^\inv(a^\inv), \; \alpha^\inv)( c \otimes a^\inv, \; \delta_{c,a^\inv} \epsilon \alpha ) \\
 = & ( \alpha^\inv(a^\inv) \otimes \alpha^\inv(c \otimes a^\inv), \; \delta_{\alpha^\inv(a^\inv),\alpha^\inv(c \otimes a^\inv) } \delta_{c,a^\inv} \alpha \epsilon) \\
 = & ( \alpha^\inv( a^\inv \otimes ( c \otimes a^\inv ) ) , \; \alpha^\inv \delta_{a^\inv, c \otimes a^\inv} \delta_{c, a^\inv} \alpha \epsilon )
\end{align*}
for some $c \in H$.

Note that $a^\inv \otimes ( c \otimes a^\inv) = (a \otimes c^\half )^2$. We assume $\iota \in \lambda$. Hence
\[ 1 = a \otimes c^\half \]
for some $c \in H$ and hence $a = c^{-1/2} \in H$. This implies $\lambda = (\alpha^\inv(H), \epsilon) \subseteq J$.
\end{proof}

\begin{corollary} \label{cor:frobenius_lines}
Any two distinct points are contained in at most one line.
\end{corollary}

\begin{lemma} \label{lem:correct_lines}
Fix distinct involutions $i,j \in J$ and suppose $\ell_{ij}$ exists in $\Lambda$. Then
\begin{align*}
\ell_{ij} = & \{ k \in J : ij \in kJ \} \\
 = & \{ k \in J : (ij)^k = (ij)^\inv \}.
 \end{align*}
\end{lemma}
\begin{proof}
  We may assume that $\ell_{ij} = H \times \{ \epsilon \}$ and $ij = (c,1)$ for some $c \in H \setminus \{ 1 \}$. 
  The second equality is easy and therefore we only show the first equality.
  
  We first show $\ell_{ij} \subseteq \{ k \in J : ij \in kJ \}$:
  Take $d \in H \setminus \{ 1 \}$. Then
  \[ (d, \epsilon)(c, 1)(d,\epsilon) = (d, \epsilon) ( c \otimes d, \delta_{c,d} \epsilon ) = ( d \otimes (c \otimes d )^\inv , \delta_{d, (c \otimes d)^\inv} \delta_{c,d} ). \]
  The Frobenius complement $H$ is abelian and therefore
  \[ ( d \otimes (c \otimes d )^\inv , \delta_{d, (c \otimes d)^\inv} \delta_{c,d} ) = (c^\inv, 1). \]
  This shows $(c,1)^{(d,\epsilon)} = (c,1)^\inv$ and hence $\ell_{ij} \subseteq \{ k \in J : ij \in kJ \}$.
  
  We now show $\supseteq$ for the first equality:
   Suppose $(c,1) = (a, \epsilon)(b, \epsilon)$. We have to show that $a$ is an element of $H$. We have
  \[ (c,1) = (a, \epsilon)(b, \epsilon) = ( a \otimes  b^\inv , \delta_{a,b^\inv}) \]
  and hence $a^\half b^\inv a^\half = a \otimes b^\inv = c \in H$ and $\delta_{a,b^\inv} = \id$. By \cref{prop:twisted_subgroup} this implies
  \[ b^{-\half} a^\half (a \otimes b^\inv ) ^{-\half} = 1. \]
  Therefore $b^{-\half}a^\half = c^\half$ and since $c = a^\half b^\inv a^\half$, this implies $a^\half b^{-\half} = c^\half$.
  Hence
  \[ c^\half = b^{-\half}a^\half = (a^\half b^{-\half})^{a^\half} = (c^\half)^{a^\half} \]
  and therefore $a^\half \in \Cen(c) = H$.
\end{proof}

\begin{lemma}\label{lem:frob_normalizer}
  Suppose $(a,\alpha) \in N_\calG(\lambda_0)$. Then $a \in H$ and $\alpha$ normalizes $H$.
\end{lemma}
\begin{proof}
Given $c \in H$ we have
\begin{align*}
  & (a,\alpha)^\inv (c, \epsilon) (a, \alpha) \\
  =& (\alpha^\inv (a^\inv), \alpha^\inv )( c\otimes a^\inv, \delta_{c,a^\inv}\epsilon \alpha) \\
  =& (\alpha^\inv(a^\inv) \otimes \alpha^\inv(c \otimes a^\inv), \delta_{\alpha^\inv(a^\inv), \alpha^\inv(c \otimes a^\inv)} \alpha^\inv \delta_{c,a^\inv} \alpha \epsilon ) \\
  =& (\alpha^\inv(a^\inv \otimes (c \otimes a^\inv)), \alpha^\inv \delta_{a^\inv, c\otimes a^\inv} \delta_{c,a^\inv}\alpha \epsilon ) \\
  =& (\alpha^\inv(a^\inv \otimes (c \otimes a^\inv)), \epsilon ) 
\end{align*}
We have $(1,\epsilon) \in \lambda_0$ and therefore
\[ 1 = a^\inv \otimes ( c_0 \otimes a^\inv ) \]
for some $c_0 \in H$.
Note that
\[ a^\inv \otimes ( c_0 \otimes a^\inv ) = (a^\ihalf c_0^\half a^\ihalf )^2 = (a^\inv \otimes c_0^\half)^2 \]
and therefore $1 = a^\inv \otimes c_0^\half$. This shows $a = c_0^\half \in H$.

Moreover, $\alpha^\inv(a^\inv \otimes (c \otimes a^\inv)) \in H$ for all $c \in H$ and hence $\alpha$ normalizes $H$.
\end{proof}

\begin{proposition}\label{prop:quasidirect_normalizer}
$N_G(\lambda_0)  \cap J^{\cdot 2} = \lambda_0^2$.
\end{proposition}
\begin{proof}
Fix $a \neq b$ in $L$ such that $(a,\epsilon)(b,\epsilon) = (a \otimes b^\inv, \delta_{a,b^\inv}) \in N_\calG(\lambda_0)$.
By \cref{lem:frob_normalizer} we have $a \otimes b^\inv \in H$ and $\delta_{a,b^\inv}$ normalizes $H$.

By \cref{prop:twisted_subgroup} the latter is equivalent to
\[ b^\ihalf a^\half (a^\half b^\inv a^\half )^\ihalf \in H. \]
Since $a^\half b^\inv a^\half = a \otimes b^\inv \in H$, this implies $b^\half a^\half \in H$ 
and therefore $a^\half b^\ihalf = a^\half b^\inv a^\half (b^\ihalf a^\half)^\inv \in H$.

This shows
\[ b^\ihalf a^\half = (a^\half b^\ihalf)^{b^\half} \in H \cap H^{b^\half}. \]
Since $H$ is malnormal in $G$, this implies $b^\half \in H$ and $a^\half = (a^\half b^\ihalf)b^\half \in H$.
\end{proof}

\begin{proposition} \label{prop:frobenius_faithfully}
Suppose $i,j,k \in J$ are pairwise distinct such that the lines $\ell_{ij}$ and $\ell_{ik}$ exist in $\Lambda$ and assume that $i,j,k$ are not collinear. Then $\Cen(i,j,k) = 1$.
\end{proposition}
\begin{proof}
  Let $i = (1, \epsilon)$ and fix $j = (a, \epsilon) \in J \setminus \{ i \}$. We already know $\Cen(i) = 1 \times \calA$.
  Now fix $(1, \beta) \in \Cen(i) \cap \Cen(j)$. Then
  \[ (\beta(a), \epsilon \beta) = (1,\beta)(a,\epsilon) = (a,\epsilon)(1, \beta) = (a, \epsilon \beta ). \]
  Therefore $\beta \in \Cen_\calA(a) = 1 \times \Cen_G(a)$ and hence $\Cen(i,j) = 1 \times \Cen_G(a)$. This shows the claim because $G$ is a Frobenius group.
\end{proof}

\begin{proof}[Proof of \cref{thm:frobenius_mock hyperbolic reflection}]
We start by checking conditions a) and b) of \cref{prop:partial_mock hyperbolic reflection}. a) follows from \cref{cor:frobenius_lines} and \cref{lem:correct_lines}. b) is part c) of \cref{lem:quasidirect_involutions}.

Now \cref{prop:quasidirect_normalizer} and \cref{lem:mock hyperbolic reflection_equivalences} imply that $(J,\Lambda)$ is a partial mock hyperbolic reflection space.

If the Frobenius group is full, then it is clear from the definition of $\lambda_0$ that all lines exist and hence that $J$ forms a mock hyperbolic reflection space.

The final statement is \cref{prop:frobenius_faithfully}.
\end{proof}

\section{Mock hyperbolic reflection spaces in groups of finite Morley rank}

We now turn to the finite Morley rank setting. We refer the reader to \cite{borovik-nesin} \cite{poizat} for a general introduction to groups of finite Morley rank. 

\begin{convention} In the context of finite Morley rank we say that a definable property $P$ holds for  \emph{Morley rank $k$ many elements} if the set defined by $P$ has Morley rank $k$. In a slight abuse, we may also say that $P$ holds for
\emph{generically many} elements of a definable set $S$ if the set of elements in $S$ not satisfying $P$ has smaller Morley rank than $S$.
\end{convention}

We will repeatedly make use of the following:
\begin{proposition}[Exercise 11 and 12 on p. 72 of \cite{borovik-nesin}] \label{no-involutions}
If $G$ is a group of finite Morley rank and $G$ does not contain an involution, then $G$ is uniquely $2$-divisible.
\end{proposition}

Now let $G$ be a group of finite Morley rank and let $Q$ be a conjugacy class of involutions such that $\Mdeg(Q)=1$. Moreover, we assume that $\Lambda \subseteq \calP(Q)$ is a $G$-invariant definable family of subsets of $Q$ such that each $\lambda \in \Lambda$ is of the form
\[ \lambda = \{ k \in Q : ij \in kQ \} \]
for any  $i\neq j \in \lambda$.
\begin{definition}
We call $(Q,\Lambda)$ a \emph{generic mock hyperbolic reflection space} if $(Q,\Lambda)$ is a partial mock hyperbolic reflection space and for each  $i\in Q$ the set \[\{j\in Q\colon \ell_{ij}\in\Lambda\}\] is generic in $Q$.
\end{definition}

\begin{remark}\label{rem: basic generic}
 Let $(Q,\Lambda)$ be a generic mock hyperbolic reflection space.
\begin{enumerate}
\item The condition in the above definition is equivalent to the statement that
\[ \{ (i,j) \in Q^2: i \neq j, \ell_{ij} \text{ exists in } \Lambda \} \subseteq Q \times Q \]
is a generic subset of $Q^2$.
\item Write
\[ \Lambda(k) = \{ \lambda \in \Lambda : \MR(\lambda) = k \}. \]
Fix $i \in Q$ and set $B_{(k)}(i) = \{ j \in Q \setminus \{ i \} : \ell_{ij} \in \Lambda(k) \}$.
Since $\Mdeg(Q) = 1$ there is exactly one $k \leq n$ such that $B_{(k)}(i)$ is a generic subset of $Q$.
In that case  $(Q,\Lambda(k))$ is a generic mock hyperbolic reflection space.  Hence we may \textbf{assume from now on that all lines in $\Lambda$ have the same Morley rank}.
\item If $(Q,\Lambda)$ is a generic mock hyperbolic reflection space of finite Morley rank in which all lines have Morley rank $k$, then we have $\MR(\Lambda)=2n-2k$ and $\Mdeg(\Lambda) = 1$ for $n=\MR(Q)$. The set of translations
\[ S = \{ \sigma \in Q^\squared \setminus \{ 1 \} : \ell_\sigma \text{ exists in } \Lambda \} \cup \{ 1 \} \]
has Morley rank $2n-k$ and Morley degree $1$.
\end{enumerate}
\end{remark}

If $X$ and $Y$ are definable sets, then we write $X \approx Y$ if $X$ and $Y$ coincide up to a set of smaller rank,
i.e. if the sets $X$, $Y$, and $X \cap Y$ all have the same Morley rank and Morley degree. This defines an equivalence relation on the family of definable sets. One important property of this equivalence relation is the following:
\begin{proposition}[Lemma 4.3 of \cite{wagner}]\label{stabilizer_lemma}
Let $G$ be a group acting definably on a set $X$ in an $\omega$-stable structure. Let $Y$ be a definable subset of $X$ such that $gY \approx Y$ for all $g \in G$. Then there is a $G$-invariant set $Z \subseteq X$ such that $Z \approx Y$.
\end{proposition}

By \cref{prop:BHNeumann} a mock hyperbolic reflection space consists of one line if and only if the set of translations forms a normal subgroup.
For generic mock hyperbolic reflection spaces the following will be shown in this section:
\begin{theorem} \label{thm:generic_mock hyperbolic reflection} \label{thm:rank_ineq}
Suppose $(Q,\Lambda)$ is a generic mock hyperbolic reflection space such that $Q$ has Morley rank $\MR(Q) = n$.
Assume that $\Lambda$ consists of more than one line and that all lines $\lambda \in \Lambda$ are infinite and of Morley rank $\MR(\lambda) = k$. Then $n \geq 2k+1$.

If $n = 2k+1$, then the translations almost form a normal subgroup: $G$ has a definable connected normal subgroup $N$ of Morley rank $\MR(N) = 2n-k$ such that $N \approx S$. Moreover, $\MR(N \cap \Cen(i)) = n-k$ for any involution $i \in Q$.
\end{theorem}

For the remainder of this section we assume that $(Q,\Lambda)$ is a generic mock hyperbolic reflection space in a group of finite Morley rank $G$ such that $(Q,\Lambda)$ satisfies the assumptions in \cref{thm:generic_mock hyperbolic reflection}, i.e. $Q$ has Morley rank $\MR(Q) = n$ and Morley degree $\Mdeg(Q) = 1$, all lines $\lambda \in \Lambda$ are infinite and of Morley rank $k$, and $\Lambda$ consists of more than one line. In particular, $n > k \geq 1$.

Note that we do not state any assumption about the Morley degree of lines.

\subsection{Generic projective planes}

\begin{definition}\label{def:generic_projective_plane}
	A definable subset $X \subseteq Q$ is a \emph{generic projective plane} if
	\begin{enumerate}
		\item $\MR(X) = 2k$ and $\Mdeg(X) = 1$, and
		\item $\MR(\Lambda_X) = 2k$ and $\Mdeg(\Lambda_X) = 1$,
	\end{enumerate}
	where $\Lambda_X$ is the set of all lines $\lambda \subseteq Q$ such that $\MR(\lambda \cap X) =k$.
\end{definition}

The next lemma follows from easy counting arguments.
\begin{lemma} \label{lem:approx_plane}
  Let $X \subseteq Q$ be a definable set of Morley rank 2k and Morley degree 1. The following are equivalent:
  \begin{enumerate}
    \item $X$ is a generic projective plane,
    \item $\MR(\Lambda_X) \geq 2k$,
    \item The set of $x\in X$ such that $\MR( \{ \lambda \in \Lambda_X : x \in \lambda \}) = k$ is generic in $X$.
  \end{enumerate}
\end{lemma}

If $X$ is a definable set, then we say that a property $P$ holds for rank $k$ many $x \in X$ (resp. generically many $x \in X$) if the set $\{ x \in X : P(x) \text{ holds} \}$ has rank $k$ (resp. is generic in $X$). 

\begin{lemma} \label{lem:plane_normalizing}
Suppose $X \subseteq Q$ is a generic projective plane. Then set of $x\in X$ such that $X^x \approx X$  is generic in  $X$.
\end{lemma}
\begin{proof}
Let $\lambda \in \Lambda_X$ be a line. For $i \in \lambda$ set 
\[ \lambda_i = \{ j \in \lambda: ij \in (\lambda^{\cdot 2})^0 \} = i(\lambda^{\cdot 2})^0.\]
Then $\{ \lambda_i : i \in \lambda \}$ is a partition of $\lambda$ into set of rank $k$ and degree $1$.
Moreover, we have $(\lambda_i)^i = \lambda_i$ for all $i \in \lambda$. In particular, if $\lambda_i \cap X \approx \lambda_i$ and $i \in X$, then $(\lambda_i \cap X)^i \cap X \approx \lambda_i$.

Hence for all $\lambda \in \Lambda_X$ the set
\[ X_\lambda = \{ x \in \lambda \cap X : \MR((\lambda \cap X)^x \cap X) = k \} \]
has Morley rank $k$. Moreover, each $x \in X$ is contained in at most rank $k$ many lines in $\Lambda_X$ and hence is contained in at most rank $k$ many sets $X_\lambda$.

We have $\MR(\Lambda_X) = 2k$ and hence the set
\[ \{ (x,\lambda) \in X \times \Lambda_X : x \in X_\lambda \} \]
has Morley rank $3k$. Since $\MR(X) = 2k$ this implies that the set of $x \in X$ contained in rank $k$ many sets $X_\lambda$ is generic in $X$.

Now if $x \in X_\lambda$ for rank $k$ many $\lambda$, then
\[ X^x \cap X \supseteq ( \bigcup_{\lambda : x \in X_\lambda} \lambda \cap X)^x \cap X 
  = \bigcup_{\lambda : x \in X_\lambda} (\lambda \cap X)^x \cap X \]
must have Morley rank $2k$ and hence $X^x \approx X$.
\end{proof}

\begin{lemma}
If $X \subseteq Q$ is a generic projective plane and $Z \subseteq Q$ is a definable subset with $X \approx Z$, then $Z$ is a generic projective plane.
\end{lemma}
\begin{proof}
	For $x \in X$ put $\Lambda_x = \{ \lambda \in \Lambda_X : x \in \lambda \}$. If $\MR( \Lambda_x ) = k$, then $B(x) = \bigcup \Lambda_x \approx X$. In particular, $B(x) \approx Z$ for a generic set of $x \in X \cap Z$. If $B(x) \approx Z$, then $\Lambda_x \cap \Lambda_Z$ must have Morley rank $k$. Hence it follows from \cref{lem:approx_plane} that $Z$ must be a generic projective plane.
\end{proof}

\begin{lemma} \label{line_counting}
	Let $H \leq G$ be a definable subgroup such that $\MR( H \cap Q) = 2k$ and $\Mdeg( H \cap Q )=1$. Then $\MR(\{ \lambda : \lambda \text{ is a line s.t. } \MR(\lambda \cap H \cap Q)=k \} ) < 2k$, i.e. $H \cap Q$ does not form a generic projective plane.
\end{lemma}
\begin{proof}
	This is proved in the same way as Proposition 11.71 of \cite{borovik-nesin}.
	Put $Z = H \cap Q$ and set $\Lambda_Z = \{ \lambda \in \Lambda : \MR(\lambda \cap Z) = k \}$.
	
	Assume towards contradiction that $\MR(\Lambda_Z) \geq 2k$. Then $Z$ is a generic projective plane and hence $\MR(\Lambda_Z)=2k$ and $\Mdeg(\Lambda_Z) = 1$.
	
	Let $\lambda \in \Lambda_Z$ be a line. By \cref{line_lemma} the family $( \lambda^i : i \in Z \setminus \lambda )$ consists of Morley rank 2k many lines which do not intersect $\lambda$. Hence the set $\{ \delta \in \Lambda_Z : \lambda \cap \delta = \emptyset \} \subseteq \Lambda_Z$ is a generic subset of $\Lambda_Z$.
	
	We aim to find a line which intersects Morley rank 2k many lines contradicting $\Mdeg(\Lambda) = 1$. For $x \in Z$ set $\Lambda_x = \{ \lambda \in \Lambda_Z : x \in \lambda \}$ and set $B(x) = \bigcup \Lambda_x \cap Z \subseteq Z$. Note that
	$\MR(B(x)) = \MR(\Lambda_x)+ k $
	 and hence $\MR( \Lambda_x ) \leq k$ for all $x \in Z$.
	 Since each $\lambda \in \Lambda$ contains Morley rank $k$ many points, we must have $\MR(\Lambda_x) = k$ for a generic set of $x \in Z$.
	 
	 Fix $x_0 \in Z$ such that $\Lambda_{x_0}$ has Morley rank 2k. Then $B(x_0) \subseteq Z$ is generic and hence $\MR(\Lambda_x) = k$ for a generic set of $x \in B(x_0)$. Since $B(x_0) = \bigcup \Lambda_{x_0}$, we can find a line $\lambda \in \Lambda_{x_0}$ such that $\MR(\Lambda_x) = k$ for a generic set of $x \in \lambda$. But then $\lambda$ intersects Morley rank 2k many lines in $\Lambda_Z$.
\end{proof}

\begin{proposition} \label{no-plane}
$Q$ does not contain a generic projective plane $X$.
\end{proposition}
\begin{proof}
	Assume $X \subseteq Q$ is a generic projective plane and
	put \[H = N_G^\approx(X) = \{ g \in G: X^g \approx X \}.\] By \cref{lem:plane_normalizing} the set $X \cap H$ is generic in $X$. Hence $\MR(H \cap Q) = 2k$.	
	 This contradicts \cref{line_counting}.
\end{proof}

\subsection{A rank inequality and a normal subgroup}
A line $\lambda \in \Lambda$ is called \emph{complete for some} $i \in Q \setminus \lambda$ if the set
$\{ j \in \lambda : \ell_{ij}\in\Lambda \}  $
is a generic subset of $\lambda$.
\begin{definition} Let $(i,j,k)$ be a triple of pairwise distinct involutions in $Q$.
\begin{itemize}
\item $(i,j,k)$ is \emph{good} if $\ell_{ij}$, $\ell_{jk}$ exist and
 $\ell_{ij}$  is complete for $k$.
\item $(i,j,k)$ is \emph{perfect} if 
$\ell_{ij}, \ell_{jk}$ exist and  $\{j'\in \ell_{jk}\colon\ell_{ij'}\in\Lambda \mbox{ is complete for } k'=j'jk \}$  is generic  in $\ell_{jk}$.
\end{itemize}
\end{definition}

\begin{lemma} A generic triple $(i,j,k) \in Q^3$ is good. In particular, for any $i\in Q$ a generic element of $\{i\}\times Q^2$ is good.
\end{lemma}
\begin{proof} 
Fix $i \in Q$ and put $B(i) = \{j\in Q\colon \ell_{ij}\in\Lambda\}$. Then $B(i)$ is a generic subset $Q$.
Now fix $k \in Q \setminus \{ i \}$. We aim to show that $(i,j,k)$ must be good for  generically many $j \in Q \setminus \{i,k\}$.

Note that $B(i)$ and $B(k)$ are generic subsets of $Q$. Therefore $B(i) \cap B(k)$ must be generic in $B(i)$ and $B(i) \setminus B(k)$ is not generic in $B(i)$. Note that
\begin{align*}
B(i) \cap B(k) &= \bigcup_{ \lambda \in \Lambda_i} (\lambda \cap B(k) ) \quad \text{and} \\
B(i) \setminus B(k) &= \bigcup_{ \lambda \in \Lambda_i} (\lambda \setminus B(k) ).
\end{align*}
Since $\MR(Q) = \MR( \Lambda_i) + k$ the set
\[ \{ \lambda \in \Lambda_i : \MR( \lambda \setminus B(k) ) < k \} \]
must be generic in $\Lambda_i$.

Hence $\lambda \cap B(k) \approx \lambda$ for generically many $\lambda \in \Lambda_i$. Moreover, if $\lambda \cap B(k) \approx \lambda$ for some $\lambda \in \Lambda_i$ and $j$ is contained in $\lambda \setminus \{i,k\}$, then $(i,j,k)$ is good.
The last sentence follows since all elements in $Q$ are conjugate.
\end{proof}

\begin{proposition}\label{prop:perfect}
A generic triple $(i,j,k) \in Q^3 $ is perfect and for any $i\in Q$ a generic element of $\{i\}\times Q^2$ is perfect.
\end{proposition}
\begin{proof}
Since $Q$ is a generic mock hyperbolic reflection space, the set $U = \{ (j,k) : jk \in S \setminus \{ 1 \} \} \subseteq Q^2$ is generic in $Q^2$. For $\sigma \in S$ put $U_\sigma = \{ (j,k) : jk = \sigma \}$.
Then each $U_\sigma$ has Morley rank $k$ and $U$ is the disjoint union
\[ U = \bigcup_{\sigma \in S} U_\sigma \subseteq Q \times Q. \]
Now fix $i \in Q$. A generic triple in $\{i\} \times U$ is good  and we have $\Mdeg( \{i\} \times U) = 1$. 
Since $\Mdeg(S) = 1$ this implies that for  generically many $\sigma \in S$ the set
\[ \{ (i,r,s) \in \{i \} \times U_\sigma : (i,r,s) \text{ is good} \} \]
is a generic subset of $\{i\} \times U_\sigma$. 

Moreover, if a generic triple in $\{i\} \times U_\sigma$ is good, then a generic triple in $\{i\} \times U_\sigma$ must be perfect. This proves the lemma.
\end{proof}

Now let $\mu : Q^3  \rightarrow G$ be the multiplication map and put
\begin{align*}
T &= \{ (i,j,k) \in Q^3 : \ell_{jk} \text{ exists} \}, \quad \text{and} \\
T_\perf &= \{ (i,j,k) \in Q^3 : (i,j,k) \text{ is perfect} \}.
\end{align*}
Note that $T_\perf \subseteq T$. If $(Q,\Lambda)$ is a mock hyperbolic reflection space, i.e. if all lines exist, then $T_\perf$ consists of all triples of pairwise distinct involutions in $Q$.

\begin{lemma} $\MR(\mu(T_\perf)) \geq 2n-k$.
\end{lemma}
\begin{proof}
For any $i\in Q$ the set $\{(j,k)\in Q^2\colon (i,j,k) \mbox{ is perfect}\}$ has Morley rank $2n$ by \cref{prop:perfect}. Clearly $ijk = ij'k'$ if and only if $jk = j'k'$. If $\ell_{jk}$ exists, the set $\{(j',k')\in Q^2\colon jk = j'k'\}$ has Morley rank $k$. Hence $\mu(T_\perf)$ has Morley rank at least $2n-k$.
\end{proof}

\begin{proposition} \label{prop:normal_subgroup}
Suppose $\MR(\mu(T_\perf)) = 2n-k$. Then $G$ has a definable connected normal subgroup $N$ of Morley rank $\MR(N) = 2n-k$ such that $N \approx S$. Moreover, $\MR(N \cap \Cen(i)) = n-k$ for any involution $i \in Q$.
\end{proposition}
\begin{proof}
Set $d = \Mdeg(\mu(T_\perf))$ and write $\mu(T_\perf)$ as a disjoint union
\[ \mu(T_\perf) = Y_1 \cup \dots Y_d \]
where each $Y_r$ has rank $2n-k$ and degree $1$. Put $T_i = T_\perf \cap ( \{i\} \times Q \times Q)$. Then each $T_i$ has rank $2n$ and degree $1$.
Moreover, $\mu(T_i)$ has rank $2n-k$ and degree $1$. We can find $1 \leq f \leq d$ such that
\[ \mu(T_i) \approx Y_f \]
for generically many $i \in Q$. Put $Y = Y_f$ and set $N = \Stab^\approx(Y) = \{ g \in G : gY \approx Y \}$ and note that $N$ must be a normal subgroup of $G$ because $Y$ is $G$-normal up to $\approx$-equivalence.

Now by \cref{stabilizer_lemma} there is some $Z \approx Y$ such that $N \subseteq \Stab(Z)$. In particular, $N$ has rank $\leq 2n-k$.

Let $Q_Y = \{ i \in Q : \mu(T_i) \approx Y \}$. Given $i \neq j \in Q_Y$ we have
\[ ij\mu(T_j) \approx \mu(T_i) \]
and hence $ij \in N$. Hence $Q_Y^{\cdot 2} \subseteq N$ and thus $\MR(N) = 2n-k$. Therefore $N = \Stab(Z)$ is connected.
Since $Q_Y \approx Q$ this also implies $N \approx S$.

We now show $\MR(N \cap \Cen(i))= n-k$ for any involution $i \in Q$: Fix an involution $i \in Q$. If $i \in N$, then $iQ \subseteq N$ and hence $N = iQ (N \cap \Cen(i))$ by \cref{prop:decomposition} and therefore $\MR(N \cap \Cen(i)) = n-k$.

If $i \not \in N$, then note that $iQ \cap N$ must be a generic subset of $iQ$ and therefore the conjugacy class $i^N$ is generic in $Q$. This implies that $N \rtimes \langle i \rangle$ must contain $Q$ and hence is a normal subgroup of $G$. Now argue as in the first case.
\end{proof}

For $\alpha \in \mu(T)$ we set 
\[ X_\alpha = \{i \in Q : \exists (j,k) \in Q \times Q \text{ such that } (i,j,k) \in T \text{ and } ijk = \alpha \}. \]
Note that $\MR(\mu^\inv(\alpha) \cap T ) = \MR(X_\alpha) +k$.

If $A$ and $B$ are definable sets, then we write $A \subsetsim B$ if $A$ is almost contained in $B$, i.e. if $A \cap B \approx A$.
\begin{lemma} \label{lem:ineq_perfect}
Fix a triple $(i,j,k) \in T$.
\begin{enumerate}
\item If $(i,j,k)$ is good, then $\ell_{ij} \subsetsim X_{ijk}$.
\item If $(i,j,k)$ is perfect, then $\ell_{it} \subsetsim X_{ijk}$ for generically many $t \in \ell_{jk}$. In particular, $\MR(X_{ijk}) \geq 2k$. 
\end{enumerate}
\end{lemma}
\begin{proof}
a) Since $(i,j,k)$ is good the line $\ell_{ij}$ is $k$-complete. Hence $\ell_{j'k}$ exists for generically many $j' \in \ell_{ij}$. Fix such an $j'$ and write $ij = i'j'$. Then $(i',j',k)$ is good and hence $i' \in X_{ijk}$.

b) follows immediately from a).
\end{proof}

\begin{lemma}
	Set $ l = \MR(\mu(T_\perf)) - (2n-k)$. Then $2k \leq X_\alpha \leq n-l$ for generically many $\alpha \in \mu(T_\perf)$. In particular, $n \geq 2k+l$.
\end{lemma}
\begin{proof} 
We have $\MR(\mu^\inv(\alpha) \cap T ) = \MR(X_\alpha ) + k$ and
\[ \bigcup_{\alpha \in \mu(T_\perf)}\mu^\inv(\alpha) \cap T \subseteq T \]
and $T$ has rank $3n$. Therefore a generic $\alpha \in \mu(T_\perf)$ must satisfy the inequality
\[ \MR(\mu(T_\perf)) + \MR(X_\alpha) + k \leq \MR(T) = 3n. \]
Moreover, we have $\MR(X_\alpha) \geq 2k$ by \cref{lem:ineq_perfect}. Hence
\[ 2k \leq \MR(X_\alpha) \leq \MR(T) - k - \MR(\mu(T_\perf)) = n-l \]
for generically many $\alpha \in \mu(T_\perf)$.
\end{proof}

\begin{proposition} \label{prop:ineq}
Set $ l = \MR(\mu(T_\perf)) - (2n-k)$. Then $n > 2k+l$. In particular, $n > 2k$.
\end{proposition}
\begin{proof}
Assume not. Then $n = 2k+l$ and $\MR(X_\alpha) = 2k$ for generically many $\alpha \in \mu(T_\perf)$.
Set $M = \{ \alpha \in \mu(T_\perf) : \MR(X_\alpha) = 2k \}$. This is a generic subset of $\mu(T_\perf)$.
We have

\[ \MR( \bigcup_{\alpha \in M} \mu^\inv(\alpha) \cap T ) = (2n-k+l) + 3k = 6k + 3l = 3n. \]
Hence $\bigcup_{\alpha \in M} \mu^\inv(\alpha) \cap T$ is a generic subset of $Q \times Q \times Q$.
Note that $\MR(M) = 3k+3l$. Thus we can find $\alpha \in M$ such that $\mu^\inv(\alpha) \cap T$ has rank $3k$ and contains rank $3k$ many perfect triples. Set $X = X_\alpha$ and $\Lambda_X = \{ \lambda \in \Lambda : \lambda \subsetsim X \}$.
Now \cref{lem:ineq_perfect} implies that for a generic $i \in X$ the set
\[ \{ \lambda \in \Lambda_X : i \in \lambda \} \]
has Morley rank $k$. Hence $\MR(\Lambda_X) = 2k$ and therefore a degree $1$ component of $X$ must be a generic projective plane. This contradicts \cref{no-plane}.
\end{proof}

\begin{proof}[Proof of \cref{thm:generic_mock hyperbolic reflection}]
Set $ l = \MR(\mu(T_\perf)) - (2n-k)$. By \cref{prop:ineq} we have $2k+1 = n > 2k+l$ and hence $l = 0$. Now \cref{prop:normal_subgroup} implies the theorem.
\end{proof}

\section{Frobenius groups of finite Morley rank}
We now consider Frobenius groups of finite Morley rank. If $G$ is a group of finite Morley rank and $H$ is a Frobenius complement in $G$, then $H$ is definable by Proposition 11.19 of \cite{borovik-nesin}. If $G$ splits as $G=N \rtimes H$, then $N$ is also definable by Proposition 11.23 of \cite{borovik-nesin}.

Epstein and Nesin showed that if $H<G$ is a Frobenius group of finite Morley rank and $H$ is finite, then $H<G$ splits (Theorem 11.25 of \cite{borovik-nesin}). As a consequence it suffices to consider connected Frobenius groups of finite Morley rank (Corollary 11.27 of \cite{borovik-nesin}).

Solvable Frobenius groups of finite Morley rank split and their structure is well understood \cite[Theorem 11.32]{borovik-nesin}.

\begin{lemma}
Let $H < G$ be a connected Frobenius group of finite Morley rank with Frobenius complement $H$
and let $X \subseteq H \setminus \{ 1 \}$ be a definable $H$-normal subset such that $\MR(X) = \MR(H)$.
Then $\bigcup_{b \in G}X^b \subseteq G$ is a generic subset of $G$.
\end{lemma}
\begin{proof}
Set $n = \MR(G)$ and $k = \MR(H)$.
Consider the map $\alpha: G \times X \rightarrow G, \; (b,x) \mapsto x^b$. 
If $x^b = y^c$ for $x,y \in X$, $b,c \in G$, then $bc^{-1}$ must be contained in $N_G(H) = H$. Therefore we have
\[ \alpha^{-1}(x^b) = \{ (c, x^{bc^{-1}}) \in G \times X : bc^{-1} \in H \}. \]
Hence all fibers of $\alpha$ have Morley rank $k$. This shows that $\alpha(G \times X) = \bigcup_{b \in G}X^b$ must have Morley rank $n$ and hence is a generic subset of $G$.
\end{proof}

Groups of finite Morley rank can be classified by the structure of their $2$-Sylow subgroups. In case of Frobenius groups this classification is simpler:

\begin{proposition}
Let $G$ be a connected Frobenius group of finite Morley rank with Frobenius complement $H$. Then $H$ is connected and $G$ lies in one of the following mutually exclusive cases:
\begin{enumerate}
\item  $H$ contains a unique involution, i.e. $G$ is of odd type;
\item  $G$ does not contain any involutions, i.e. $G$ is of degenerate type;
\item  $G \setminus ( \bigcup_{g \in G} H^g )$ contains involutions, i.e. $G$ is of even type.
\end{enumerate}
\end{proposition}
\begin{proof}
We first show that $H$ must be connected: If $H$ is not connected, then $\bigcup_{g \in G}(H^0\setminus \{ 1 \})^g$ and $\bigcup_{g \in G}(H \setminus H^0)^g$ would be two disjoint generic subsets of $G$. This is impossible because $G$ is connected.

If $H$ contains an involution, then Delahan and Nesin showed that this involution must be unique and moreover all involutions in $G$ are conjugate and hence $G \setminus ( \bigcup_{g \in G} H^g )$ cannot contain any involution (Lemma 11.20 of \cite{borovik-nesin}). In particular, $G$ is of odd type because the connected subgroup $H$ contains a unique involution.

If $G \setminus ( \bigcup_{g \in G} C^g )$ contains an involution, then the proof of \cite[Theorem 2]{altinel-berkman-wagner} shows that $G$ is of even type.
\end{proof}

\begin{remark} \label{rem:even-type}
If $H<G$ is of even type, then Alt{\i}nel, Berkman, and Wagner showed in~\cite{altinel-berkman-wagner} that there is a definable normal subgroup $N$ such that $N \cap H = 1$ and $N$ contains all involutions of $G$. By Lemma 11.38 of \cite{borovik-nesin} either $G = N \rtimes H$ splits or $HN/N < G/N$ is a Frobenius group of finite Morley rank. 
Now if $HN/N < G/N$ splits, then it is easy to see that $H<G$ must split. Hence a non-split Frobenius group of minimal Morley rank cannot be of even type. Therefore to show that all Frobenius groups of finite Morley rank split, it suffices to consider Frobenius groups of odd and degenerate type.
\end{remark}

\subsection{Frobenius groups of odd type}
Let $H<G$ be a connected Frobenius group of finite Morley rank and odd type. Note that $G$ contains a single conjugacy class of involutions which we denote by $J$. Moreover, $J$ has Morley degree $1$.

\begin{proposition}[Proposition 11.18 of \cite{borovik-nesin}]
Let $H<G$ be a connected Frobenius group of finite Morley rank and odd type and $J$ its set of involutions.
If $a \in J^{\cdot 2} \setminus \{1\}$ and $i \in J$, then $\Cen(a) \cap \Cen(i) = \{ 1 \}$. 
\end{proposition}

\begin{lemma}\label{lem:involutions in Frob}
Let $H<G$ be a connected Frobenius group of finite Morley rank and odd type and $J$ its set of involutions. Fix distinct involutions $i,j \in J$.
\begin{enumerate}
\item If $a \in iJ \setminus \{ 1 \}$, then $\Cen(a) \subseteq iJ$ is a uniquely $2$-divisible abelian group;
\item $iJ$ is uniquely $2$-divisible;
\item $J$ acts regularly on itself, i.e., given $i,j \in J$ there is a unique $k \in J$ such that $j = i^k$;
\item $iJ \cap jJ$ is uniquely $2$-divisible;
\item $\Cen(ij) = iJ \cap jJ$;
\item The family $\{ \Cen(a) \setminus \{1\} : a \in Q^{\cdot 2} \setminus \{ 1 \} \}$ forms a partition of $J^{\cdot 2} \setminus \{ 1 \}$.
\end{enumerate}
\end{lemma}
\begin{proof}
1) By the previous proposition we have $\Cen(a) \cap \Cen(i) = \{1\}$ for all involutions. In particular, $\Cen(a)$ does not contain an involution and hence is uniquely $2$-divisible by \cref{no-involutions}. Note that $i$ acts on $\Cen(a)$ as a fixed point free involutionary automorphism. Hence by \cref{prop:neumann_abelian} $\Cen(a)$ is abelian and inverted by $i$.

2) Fix $a = ip \in iQ \setminus \{ 1 \}$. Since $\Cen(a) \subseteq iJ$ is uniquely $2$-divisible we have $a = b^2$ for some $b = iq \in iQ$. If $a = c^2$ for another element $c = ir \in iJ$, then $ii^p = ii^r$ and hence $pr \in \Cen(i) \cap Q^{\cdot 2} = \{ 1 \}$. Thus $b=c$. 

3) Note that $j = i^p$ if and only if $ij = ii^p = (ip)^2$. Since $iJ$ is uniquely $2$-divisible, $p$ exists and is unique.

4) It suffices to show that $iJ \cap jJ$ is $2$-divisible. Given $a \in iJ \cap jJ$ we have $\Cen(a) \subseteq iJ \cap jJ$ and hence $a = b^2$ for some $b \in \Cen(a) \subseteq iJ \cap jJ$.

5) By 3) we have $\Cen(ij) \subseteq iJ \cap jJ$. Hence it remains to show that $iJ \cap jJ \subseteq \Cen(ij)$.
Given $a \in iJ \cap jJ$, $a$ is inverted by $i$ and $j$ and hence $a \in \Cen(ij)$.

6) Suppose $c \in \Cen(a) \cap \Cen(b)$ for some $c \neq 1$. Then $a,b \in \Cen(c)$ and hence $a \in \Cen(b)$ because $\Cen(c)$ is abelian. This implies $\Cen(b) \subseteq \Cen(a)$ because $\Cen(b)$ is abelian. Hence $\Cen(a) = \Cen(b)$ by symmetry. This implies 5).
\end{proof}

Given two distinct involutions $i \neq j$ in $J$ we define the line
\[ \ell_{ij} = \{ k \in Q : (ij)^k = (ij)^\inv \}. \]

\begin{lemma} \label{lem:odd_lines}
Let $H<G$ be a connected Frobenius group of finite Morley rank and odd type and $J$ its set of involutions.
 Let $i\neq j\in J$. Then $i\ell_{ij} = \Cen(ij)$.
\end{lemma}
\begin{proof}
Clearly $i\ell_{ij} \subseteq \Cen(ij)$. On the other hand we have $\Cen(ij) \subseteq iJ$ by Part 6 of \ref{lem:involutions in Frob}. 
Given $\sigma = ik \in \Cen(ij)$, we have
\[ (ji)^k = (ij)^{ik} = (ij)^\sigma = ij. \]
Therefore $(ij)^k = (ij)^\inv$ and thus $k \in \ell_{ij}$.
\end{proof}

\begin{lemma}
Let $H<G$ be a connected Frobenius group of finite Morley rank and odd type and $J$ its set of involutions. Fix  $i\neq j\in J$ and
let $p,q \in \ell_{ij}$ be distinct involutions. Then $\ell_{pq} = \ell_{ij}$.
\end{lemma}
\begin{proof}
We have $pq \in \Cen(ij)$ and hence $\Cen(pq) = \Cen(ij)$. Moreover, $ip \in \Cen(ij) = \Cen(pq)$ and hence
\[ \ell_{pq} = p\Cen(pq) = i\Cen(ij) = \ell_{ij}. \]
\end{proof}

Hence the set $J$ together with the above notion of lines satisfies conditions a) and b) of \cref{prop:partial_mock hyperbolic reflection}.

\begin{lemma}
Let $H<G$ be a connected Frobenius group of finite Morley rank and odd type and $J$ its set of involutions.
  Let $i,j \in J$ be distinct involutions and let $T$ be a subgroup of $G$ such that $\Cen(ij) \leq T \leq N_G(\Cen(ij))$. Then $T= \Cen(ij)(T \cap \Cen(i))$.
\end{lemma}
\begin{proof}
  Note that $G$ can be decomposed as $G = iJ\Cen(i)$ and put $\lambda = \ell_{ij}$.
  Given $t \in T$, we can write $t =ikg$ for (unique) elements $k \in J$, $g \in \Cen(i)$. Then
  \[ \lambda = \lambda^t = \lambda^{kg}. \]
  In particular, $i \in \lambda^k \cap \lambda$. If $\lambda^k \cap \lambda = \{i\}$, then $k=i \in \lambda$. If $\lambda^k = \lambda$, then $k \in \lambda$ by part a) of \cref{lem:basic_properties}.
  Therefore $t = ikg \in \Cen(ij)(T \cap \Cen(i))$.
\end{proof}

\begin{proposition}
Let $H<G$ be a connected Frobenius group of finite Morley rank and odd type and $J$ its set of involutions.
If $i \neq j$ are two distinct involutions in $J$, then
\[ N_G(\Cen(ij)) \cap Q^{\cdot 2} = \Cen(ij). \]
\end{proposition}
\begin{proof}
Assume there is $a \in (N_B(\Cen(ij)) \cap J^{\cdot 2} ) \setminus \Cen(ij)$ and consider the group $A = \Cen(a) \cap N_B(\Cen(ij))$. Note that $A$ is abelian and hence $H = \Cen(ij) \rtimes A$ is a solvable subgroup of $N_B(\Cen(ij))$. By the previous lemma we have $H = \Cen(ij) \rtimes (H \cap \Cen(i))$. Now Theorem 9.11 of \cite{borovik-nesin} implies that $A$ and $H \cap \Cen(i)$ are conjugate. This is impossible.
\end{proof}

\begin{theorem}\label{thm:n=2k+1}
Let $H<G$ be a connected Frobenius group of Morley rank $n$ and odd type and let $J$ be the set of all involutions in $G$.
\begin{enumerate}
\item $J$ forms a mock hyperbolic reflection space and all lines in $J$ are infinite.
\item Choose $\Lambda$ such that $(J,\Lambda)$ is a generic mock hyperbolic reflection space such that all lines are of Morley rank $k$ and set $n = \MR(J)$. If $n \leq 2k+1$, then $G$ splits.
\end{enumerate}
\end{theorem}
\begin{proof} 
a) We first show that $J$ forms a mock hyperbolic reflection space. We already know that conditions a) and b) of \cref{prop:partial_mock hyperbolic reflection} are satisfied.
Fix a line $\lambda = \ell_{ij}$. Then $\lambda^\squared = i\lambda$ by \cref{lem:basic_properties} and hence $\lambda^\squared = \Cen(ij)$ by \cref{lem:odd_lines}. Therefore $J$ forms a mock hyperbolic reflection space by \cref{lem:mock hyperbolic reflection_equivalences}.

Moreover, $\Cen(ij)$ is infinite by Proposition 1.1 of \cite{involutions-in-degenerate} and therefore all lines in $J$ must be infinite.

b) Note that if the mock hyperbolic reflection space $J$ consists of a single line, then $H<G$ splits by \cref{prop:BHNeumann}. Hence by \cref{thm:rank_ineq} we may assume $n = 2k+1$. Then again by \cref{thm:rank_ineq} $B$ has a connected normal subgroup $N$ of rank $2n-k$ such that $N \approx S$ where
\[ S = \{ \sigma \in J^{\cdot 2} \setminus \{ 1 \} : \ell_\sigma \text{ exists} \} \cup \{ 1 \} \]
is the set of translations.
Recall that $\MR(S) = 2n-k$ and $\Mdeg(S) = 1$.

On the other hand $N \cap \Cen(i) < N$ is a connected Frobenius group of finite Morley rank and hence
$\bigcup_{i \in Q} N \cap \Cen(i) \subseteq N$ is a generic subset of $N$. This contradicts $N \approx S$.
\end{proof}

As a direct consequence we get the following known corollary (which also follows from Lemma 11.21 and Theorem 11.32 of \cite{borovik-nesin}).
\begin{corollary}
Let $H<G$ be a connected Frobenius group of finite Morley rank of odd type. If $G$ has a non-trivial abelian normal subgroup, then $G$ splits.
\end{corollary}
\begin{proof}
This follows directly from \cref{prop:BHNeumann}.
\end{proof}

\begin{proposition}
Let $H<G$ be a connected non-split Frobenius group of finite Morley rank and odd type and let $(J, \Lambda)$ be the associated mock hyperbolic reflection space. If generic lines have Morley rank $1$, then $G$ is a non-split sharply $2$-transitive group of characteristic $\neq 2$.
\end{proposition}
\begin{proof}
Set $n = \MR(J)$.
The set of translations has Morley rank $2n-1$ and is not generic in $G$. On the other hand, $\Cen(i)$ acts on $J \setminus \{ i \}$ without fixed points. Therefore $\MR(\Cen(i) \leq 2n$. Hence $G = iJ\Cen(i)$ must have Morley rank $2n$ and $\Cen(i)$ has Morley rank $n$. This implies that $\Cen(i)$ acts regularly on $J \setminus \{ 1 \}$ and hence $G$ is a sharply $2$-transitive group.
\end{proof}

\begin{remark}
We will see in \cref{cor:simplicity} that the group in the above proposition must in fact be simple.
\end{remark}

\begin{proposition}
Let $H<G$ be a connected Frobenius group of Morley rank at most $10$ and odd type. Then either $H<G$ splits or $G$ is a simple non-split sharply $2$-transitive group of Morley rank $8$ or $10$.
\end{proposition}
\begin{proof}
Assume $G$ does not split.
Suppose the set of involutions has Morley rank $n$ and the lines in the associated generic mock hyperbolic reflection space have rank $k$. Since the set of translations is not generic in $G$ we have $\MR(G) > 2n-k$. Moreover, we know $n > 2k+1$ and $k \geq 1$. This shows $\MR(G) > 2(2k+2)-k = 3k+4$. Hence if $k \geq 2$, then $\MR(G) > 10$. Therefore if $\MR(G) \leq 10$, then $k=1$ and $\MR(G) > 7$. The previous proposition and the remark show that $G$ is a simple sharply $2$-transitive group and hence $\MR(G)$ must be an even number, thus $\MR(G)$ is either $8$ or $10$.
\end{proof}

\subsection{Frobenius groups of odd type with nilpotent complement}
Delahan and Nesin showed that a sharply $2$-transitive group of finite Morley rank of characteristic $\neq 2$ with nilpotent point stabilizer must split (Theorem 11.73 of \cite{borovik-nesin}). We will show that the same is true for a Frobenius group of odd type if the lines in the associated mock hyperbolic reflection geometry are strongly minimal or if there is no interpretable bad field of characteristic $0$.

We fix a connected Frobenius group $H<G$ of finite Morley rank of odd type and we denote the set of involutions by $J$.
By \cref{thm:n=2k+1} $J$ forms a mock hyperbolic reflection space with infinite lines. Recall that $H = \Cen(i)$ if $i$ is the unique involution in $H$. If $\lambda$ is a line containing $i$, then $N_G(\lambda) = \lambda^\squared \rtimes N_H(\lambda)$ is a split Frobenius group.

\begin{lemma} \label{normalizer_index}
  If $i \neq j$ are involutions, then $\Cen(ij)$ has infinite index in $N_G(\Cen(ij))$. In particular, $N_{\Cen(i)}(\ell_{ij})$ is infinite.
\end{lemma}
\begin{proof}
  Otherwise $\bigcup_{g \in G}\Cen(ij)^g \subseteq J^{\cdot 2}$ would be generic in $G$. This is impossible since $\bigcup_{i \in J} \Cen(i)$ is generic in $G$ and the elements of $J^\squared$ do not have fixed points. Therefore $\Cen(ij)$ has infinite index in $N_G(\Cen(ij))$.
  
Now $N_G(\Cen(ij)) = N_G(\ell_{ij})$ and $N_G(\ell_{ij}) \cap J = \ell_{ij}$ forms a mock hyperbolic reflection space (consisting of one line). Therefore $N_G(\Cen(ij)) = i\ell_{ij}N_{\Cen(i)}(\ell_{ij})$ and thus $N_{\Cen(i)}(\ell_{ij})$ is infinite.
\end{proof}

If the point stabilizer in a sharply $2$-transitive group of characteristic $\neq 2$ with planar maximal near-field contains an element $g \not \in \{1,i \}$ such that $g$ normalizes all lines containing $i$, then by \cite{near-domains} the sharply $2$-transitive group splits. We are going to prove a similar result for Frebenius groups of finite Morley rank of odd type.

If $A$ is a group, then we write $A^* = A \setminus \{ 1 \}$.
\begin{lemma}
Let $\lambda$ be a line containing $i \in J$ and fix a definable solvable subgroup $A \leq N_{\Cen(i)}(\lambda)$. Then $A^*i\lambda \cup \{1\} = A^\lambda$.
\end{lemma}
\begin{proof}
Note that $H = \lambda^{\cdot 2} \rtimes A$ is a solvable Frobenius group of finite Morley rank. By Theorem 11.32 of \cite{borovik-nesin} we have $H = \lambda^{\cdot 2} \cup \bigcup_{j \in \lambda} A^j$. This proves the lemma.
\end{proof}

\begin{proposition} \label{prop:splitting}
Let $\Lambda$ be a set of lines on $J$ such that $(J,\Lambda)$ forms a generic mock hyperbolic reflection space. Suppose there exists a definable infinite solvable normal subgroup $A \trianglelefteq \Cen(i)$ such that $A \leq \bigcap_{ \lambda \in \Lambda_i } N_{\Cen(i)}(\lambda)$. Then $H<G$ splits. 
\end{proposition}
\begin{proof}
We may assume that all lines in $\Lambda$ have Morley rank $k$. Since $i$ is central in $\Cen(i)$, we may also assume that $i \in A$.

Now set $J_i = \bigcup_{ \lambda \in \Lambda_i} \lambda =  \{ j \in J \setminus \{i \} : \ell_{ij} \text{ exists} \} \cup \{ i \}$.
By the previous lemma we have $A^*i\lambda \cup \{1\} = A^\lambda$ for all $\lambda \in \Lambda_i$.
Hence we have
\[ A^* iJ_i \cup \{1\} = \bigcup_{\lambda \in \Lambda_i}A^*i\lambda \cup \{1\} 
  = \bigcup_{\lambda \in \Lambda_i}A^\lambda = A^{J_i}. \]
Therefore
\[ A^*iJ \cup \{1\} \approx A^*iJ_i \cup \{1\} = A^{J_i} \approx A^J = A^{\Cen(i)J} = A^G. \]
Put $N = \Stab^\approx(A^G)$. Then $A \leq N$ and $N \trianglelefteq G$ is a normal subgroup. Hence $A^G \leq N$.
Now \cref{stabilizer_lemma} implies that $A^G \approx N$. 
Note that 
\[ \MR(N) = \MR(J) + \MR(A). \]
Moreover, $J \subseteq N$, therefore $J^{\cdot 2} \subseteq N$ and thus $\MR(N) \geq 2n-k$. Note that $A$ acts fixed point free on any line $\lambda \in \Lambda_i$ and therefore $\MR(A) \leq k$. In conclusion
\[ n-k \leq \MR(N) - \MR(J) = \MR(A) \leq k \]
and therefore $n \leq 2k$. Now \cref{prop:ineq} implies that $H<G$ splits.
\end{proof}

\begin{corollary}
Let $H<G$ be a connected Frobenius group of finite Morley rank and odd type. If $H$ is a minimal group, i.e., if $H$ does not contain an infinite proper definable subgroup, then $H<G$ splits.
\end{corollary}
\begin{proof}
The assumptions and \cref{normalizer_index} imply that $N_{\Cen(i)}(\lambda) = \Cen(i)$ holds for all $i$ in $J$ and $\lambda \in \Lambda_i$. If $H = \Cen(i)$, then $H = \bigcap_{\lambda \in \Lambda_i} N_H(\lambda)$ and $H$ is abelian. 
Therefore \cref{prop:splitting} implies that $H<G$ splits.
\end{proof}

We can use Zilber's Field Theorem to find interpretable fields in Frobenius groups of odd type. 

\begin{proposition}[Theorem 9.1 of \cite{borovik-nesin}] \label{prop:field-theorem}
Let $G = A \rtimes H$ be a group of finite Morley rank where $A$ and $H$ are infinite definable abelian subgroups and $A$ is $H$-minimal, i.e., there are no definable infinite $H$-invariant subgroups. Assume that $H$ acts faithfully on $A$. Then there is an interpretable field $K$ such that $A \cong K_+$, $H \leq K^*$, and $H$ acts by multiplication.
\end{proposition}

Let $\lambda = \ell_{ij}$ be a line. Then $N_{\Cen(i)}(\lambda)$ is infinite and acts on $\lambda^\squared = \Cen(ij)$ by conjugation. Take a minimal subgroup $A \leq N_{\Cen(i)}(\lambda)$. Since the action of $A$ on $\Cen(ij)$ has no fixed points we can find an infinite $A$-minimal subgroup $B \leq \Cen(ij)$ on which $A$ acts faithfully. Moreover, $B$ must be abelian because $\Cen(ij)$ is an abelian group. Hence by \cref{prop:field-theorem} there is an interpretable field $K$ such that $B \cong K_+$, $A \leq K^*$, and $A$ acts by multiplication.

In particular, if the line $\lambda$ is strongly minimal, then $K$ is strongly minimal and $A \cong K^*$.

If $A$ is a proper subgroup of $K^*$, then $K$ is a bad field, i.e., an infinite field of finite Morley rank such that $K^*$ has a proper infinite definable subgroup. By \cite{bad-fields} bad fields of characteristic $0$ exist. However, it follows from work of Wagner \cite{wagner-fields} that if $\chara(K) \neq 0$, then $K^*$ is a good torus, i.e., every definable subgroup of $K^*$ is the definable hull of its torsion subgroup. We refer to \cite{good-tori} for properties of these good tori.

\begin{theorem}
Let $H<G$ be a connected Frobenius group of finite Morley rank and odd type. Fix $\Lambda$ such that $(J,\Lambda)$ is a generic mock hyperbolic reflection space.
Moreover, assume that $H$ has a definable nilpotent normal subgroup $N$ such that $N \cap N_H(\lambda)$ is infinite for all $\lambda \in \Lambda_i$.

If all lines in $\Lambda$ are strongly minimal or if $G$ does not interpret a bad field of characteristic $0$, then $H<G$ splits.
\end{theorem}
\begin{proof}
We may assume that $N$ is connected.
Let $T$ be a maximal good torus in $N$. As a consequence of the structure of nilpotent groups of finite Morley rank (Theorem 6.8 and 6.9 of \cite{borovik-nesin}) $T$ must be central in $N$. By Theorem 1 of \cite{good-tori} any two maximal good tori are conjugate. Therefore $T$ is the unique maximal good torus in $N$. Since a connected subgroup of a good torus is a good torus, the assumptions (and the previous discussion) imply that $N_H(\lambda) \cap T$ is infinite for all lines $\lambda \in \Lambda_i$. By Lemma 2 of \cite{good-tori} the family $\{ N_H(\lambda) \cap T : \lambda \in \Lambda_i \}$ is finite. Hence after replacing $\Lambda$ by a generic subset $\Lambda' \subseteq \Lambda$ we may assume that $\{ N_H(\lambda) \cap T : \lambda \in \Lambda_i \}$ consists of a unique infinite abelian normal subgroup of $H$. Now \cref{prop:splitting} implies that $H<G$ splits.
\end{proof}

\subsection{Frobenius groups of degenerate type}
We now use mock hyperbolic spaces to study Frobenius groups of finite Morley rank and degenerate type. 
A geometry with similar properties, but defined on the whole group, was used by Fr\'econ in  his result on the non-existence of bad groups of Morley rank $3$.

\begin{lemma}
Let $H<G$ be a connected Frobenius group of Morley rank $n$ and of degenerate type. Suppose the Frobenius complement $H$ is abelian and of Morley rank $k$. Then $n \geq 2k+1$ and if $n = 2k+1$, then $G$ contains a definable normal subgroup $N$ of Morley rank $k+1$.
\end{lemma}
\begin{proof}
Note that $G$ is uniquely $2$-divisible and hence $a \otimes b = a^\half b a^\half$ defines a K-loop structure on $G$. Let $L = (G, \otimes)$ denote the corresponding K-loop and set $\calA = G \times \langle \epsilon \rangle < \Aut(L)$ where $\epsilon$ is given by inversion. Now let $\calG$ be the quasidirect product $\calG = L \rtimes_Q \calA$.

By \cref{thm:frobenius_mock hyperbolic reflection} the involutions $J$ in $\calG$ form a partial mock hyperbolic reflection space and since $\bigcup_{g \in G} H^g \subseteq G$ is a generic subset of $G$ the involutions must form a generic mock hyperbolic reflection space.
Moreover, $\MR(J) = n$ and each line has Morley rank~$k$. 
Now the lemma follows from \cref{thm:generic_mock hyperbolic reflection}.
\end{proof}

\begin{theorem}
Let $H<G$ be a connected Frobenius group of Morley rank $n$ and of degenerate type. Suppose the Frobenius complement $H$ is abelian and of Morley rank $k$. Then $n \geq 2k+1$.

If $n = 2k+1$, then $G$ splits as $G = N \rtimes H$ for some definable connected normal subgroup $N$ of Morley rank $k+1$.
Moreover, if $N$ is solvable, then there is an interpretable field $K$ of characteristic $\neq 2$ such that $G = K_+ \rtimes H$, $H \leq K^*$, and $H$ acts on $K_+$ by multiplication.
\end{theorem}
\begin{proof}
By the previous lemma we may assume $n = 2k+1$. Then $G$ contains a definable normal subgroup $N$ of rank $k+1$ and we may assume that $N$ is connected.

Note that 
$\MR( \bigcup_{g \in G}(N \cap H)^g) = k+1 + \MR(N \cap H)$ and $\MR(N) = k+1$. Therefore $N \cap H$ must be finite. If $N \cap H$ is non-trivial, then $(N \cap H)<N$ is a connected Frobenius group and hence $N \cap H$ must be connected. Therefore $N \cap H = \{ 1 \}$. 

The semidirect product $N \rtimes H$ has rank $2k+1$ and hence is generic in $G$. Therefore $G = N \rtimes H$ splits.

Now assume that $N$ is solvable. Then $N$ is nilpotent by Theorem 11.59 of \cite{borovik-nesin} and hence $\Cen(u) \leq N$ for all $u \in N \setminus \{ 1 \}$ (\cite{borovik-nesin} Exercise on page 215). Note that $u^G$ cannot be generic in $N$ because $G$ does not contain involutions. Therefore $\MR(u^G) \leq k$ and hence $\MR(\Cen(u)) \geq k+1$. Therefore $\Cen(u) = N$ and hence $N$ is abelian.

We now show that $N$ is $H$-minimal: Let $A \leq N$ be a $H$-invariant subgroup. We may assume that $A$ is connected. Given $a \in A \setminus \{ 1 \}$ we have $\Cen(a) \cap H = \{1\}$ and therefore $a^H \subseteq A$ has rank $k$. If $A$ has rank $k$, then $a^H$ is generic in $A$ and therefore $A$ must contain an involution. This is a contraddiction.
Therefore $A = \{1\}$ or $A = N$ and hence $N$ is $H$-minimal.

By \cref{prop:field-theorem} there must be an interpretable field $K$ such that $N = K_+$, $H \leq K^*$, and $H$ acts on $N$ by multiplication.
\end{proof}

\section{Sharply 2-transitive groups of finite Morley rank}

Let $G$ be a sharply $2$-transitive group of finite Morley rank with $\chara(G) \neq 2$ and let $Q$ denote the set of involutions in $G$. By \cref{cor:mock hyperbolic reflection from 2-sharp} (or by \cref{thm:n=2k+1}) the set $Q$ forms a mock hyperbolic reflection space.

We set $n = \MR(Q)$ and $k = \MR(\Cen(ij))$ for involutions $i \neq j\in Q$. Note that $k$ does not depend on the choice of $i$ and $j$ and $k=n$ if and only if $G$ is split. By \cite{borovik-nesin}  Prop.~11.71 we have $0<2k<n$ and we will improve this inequality below.

Since $G$ acts sharply $2$-transitively on $Q$, it is easy to see that $\MR(G) = 2n$ and $\MR(Q^{\cdot 2}) = 2n-k$. Moreover, $G$ and $\Cen(ij)$ have Morley degree 1 by Lemma 11.60 of~\cite{borovik-nesin}.  


\begin{proposition}\label{prop:subgroup generated by Q^2} 
	\begin{enumerate}
		\item The set $iQ$ is indecomposable for all $i \in Q$.
		\item $\langle Q^{\cdot 2} \rangle$ is a definable connected subgroup. In particular, there is a bound $m$ such that any $g\in \langle Q^{\cdot 2} \rangle$ is a product of at most $m$ translations.
	\end{enumerate}
\end{proposition}
\begin{proof}
	a) Fix an involution $i \in Q$. The set $iQ$ is normalized by $\Cen(i)$, hence it suffices to check indecomposability for $\Cen(i)$-normal subgroups. If $H \leq G$ is a $\Cen(i)$-normal subgroup of $G$, then either $\Cen(ij) \leq H$ for all $j \in Q \setminus \{i\}$ or $H\  \cap\ \Cen(ij)$ has infinite index in $\Cen(ij)$ for all $j \in Q \setminus \{i\}$. Therefore the set $iQ = \bigcup_{j \in Q \setminus \{i\}} \Cen(ij)$ is indecomposable.
	
	b)  Since $\langle Q^{\cdot 2} \rangle = \langle iQ \rangle$, this follows from  Zilber's indecomposability theorem using~(a).
\end{proof}

\begin{remark}\label{rem:Rips-Tent}
By Proposition~\ref{prop:subgroup generated by Q^2} (b) it is easy to see that the non-split examples of sharply 2-transitive groups of characteristic $0$ constructed in \cite{rips-tent} do not have finite Morley rank.
\end{remark}

\begin{lemma}\label{lem:product with involution}
For any $g\in G\setminus Q$ the set $\{i\in Q\colon gi \mbox{ has a fixed point}\}$ is generic in $Q$.
\end{lemma}
\begin{proof}
Let $g\in G$. For any $j\in Q$ there is a unique $i_j\in Q$ swapping $j$ and $j^g$. Then $gi_j$ centralizes $j$, so has a fixed point.
If $i_j=i_k$ for some $j\neq k\in Q$, then by sharp 2-transitivity it follows that $g=i_j=i_k\in Q$.
Hence for  $g\notin Q$, the $i_j, j\in Q$, are pairwise distinct and hence $\{i_j, j\in Q\}$ has Morley rank $n$.
\end{proof}

Let $\mu: G^3\longrightarrow G$ be the multiplication map, i.e., $\mu(g_1,g_2,g_3)=g_1g_2g_3$.

\begin{lemma}\label{lem:Q^3}
	We have $\MR(Q^{\cdot 3})>\MR(Q^{\cdot 2})$. 
\end{lemma}
\begin{proof}
	Note that $\MR(Q^{\cdot 3})=\MR(iQ^{\cdot 3})\geq \MR(Q^{\cdot 2})>\MR(Q)=n$ and hence $Q$ is not a generic subset of $Q^{\cdot 3}$.
	
	For $\alpha \in Q^{\cdot 3}$ we let  $X_\alpha = \{i \in Q :  i\alpha \in Q^{\cdot 2} \}$ be the set of all involutions $i$ such that $i\alpha$ is a translation.
	By Lemma~\ref{lem:product with involution} and Remark~\ref{lem:basic} $MR(X_\alpha)<n$ for all $\alpha\in Q^{\cdot 3}\setminus Q$.
	
	Let $\MR(Q^{\cdot 3})=2n-k+l$ for some $l\geq 0$. There is a generic set of $\alpha \in Q^{\cdot 3} \setminus Q$ such that $\MR(\mu^{-1}(\alpha) )= n+k-l$. Set $X = X_\alpha$ for such an $\alpha \in Q^{\cdot 3} \setminus Q$. If $irs = \alpha$, then $\MR(\{ j \in Q: rs \in jQ\}) = k$ and hence $\MR(\mu^{-1}(\alpha))= \MR(X) + k$. Therefore we have $\MR(X) =n-l$ and hence $l\geq 1$ by \cref{lem:product with involution}.
\end{proof}

\begin{corollary} \label{cor:simplicity} Let $G$ be a non-split sharply $2$-transitive group of finite Morley rank.
	If the lines are strongly minimal, then $G$ is simple and a counterexample to the Cherlin-Zilber Conjecture.
\end{corollary}
\begin{proof}
 Since $\MR(Q^{\cdot 2})=2n-1<\MR(Q^{\cdot 3})=\MR(iQ^{\cdot 3})=\MR(G)=2n$ simplicity follows from Lemma~\ref{lem:Q^3} and Proposition~\ref{prop:subgroup generated by Q^2}.
 
 Assume towards a contradiction that  $G$ is an algebraic group over an algebraically closed field $K$. If the $K$-rank of $G$ is at least 2, then the torus contains commuting involutions, contradicting Remark~\ref{lem:basic} (iii). If the $K$-rank of $G$ is $1$, then $G$ is isomorphic to $PSL_2(K)$ and also contains commuting involutions, e.g. $x\mapsto -\frac{1}{x}$ and $x\mapsto -x$ are commuting involutions in $PSL_2(K)$.
\end{proof}

Note that a sharply $2$-transitive group of finite Morley rank in characteristic different from 2 is not a bad group in the sense of Cherlin since for any translation $\sigma\in Q^{\cdot 2}$ the group $N_G(\Cen(\sigma))=\Cen(\sigma)\rtimes N_{\Cen(\sigma)}(\Cen(\sigma))$ is solvable, but not nilpotent.

If $G$ is a sharply $2$-transitive group of finite Morley rank and $\chara(G)\neq 2$ with $n, k$ and $Q$ be as before, then by \cref{thm:n=2k+1} $G$ splits if $n\leq 2k+1$. Thus, we obtain:

\begin{corollary}
	If $G$ is a sharply $2$-transitive group, \ $\MR(G) = 6$, then $G$ is of the form $\mathrm{AGL}_1(K)$ for some algebraically closed field $K$ of Morley rank $3$.
\end{corollary}
\begin{proof}
If $\chara(G)\neq 2$, then by \cref{thm:n=2k+1} $G$ splits and the result follows from ~\cite{altinel-berkman-wagner}. If $\chara(G)=2$, then $G$ is split by~\cite{altinel-berkman-wagner} and any point stabilizer has Morley rank 3. Since the point stabilizers do not contain involutions, they are solvable by~\cite{frecon}. Now the result follows from~\cite{borovik-nesin}, Cor.~11.66. 
\end{proof}

\section{Further remarks}
A finite uniquely $2$-divisible K-loop is the same as a finite B-loop in the sense of Glauberman \cite{B-loops}. As a consequence of Glauberman's $Z^*$-Theorem \cite{Z-star} finite B-loops are solvable. 
Following Glauberman we say that a K-loop $L$ is \emph{half-embedded} in some group $G$ if it is isomorphic to a K-loop arising from a uniquely $2$-divisible twisted subgroup of $G$ as in \cref{prop:twisted_subgroup}. B-loops and uniquely $2$-divisible K-loops can always be half-embedded in some group and that group can be chosen to be finite if the loop is finite (Theorem~1 and Corollary 1 of \cite{B-loops}).
This allows us to restate Glauberman's result for twisted subgroups:

\begin{proposition}[\cite{Z-star}]
Let $G$ be a group and let $L \subseteq G$ be a finite uniquely $2$-divisible twisted subgroup. Then $\langle L \rangle$ is solvable.
\end{proposition}

As a consequence finite mock hyperbolic spaces must consist of a single line:

\begin{proposition} \label{prop:finite-mock-hyperbolic}
Suppose $Q$ forms a finite mock hyperbolic reflection space in a group $G$. Then $Q$ consists of a single line. 
\end{proposition}
\begin{proof} We may assume that $G$ acts faithfully on $Q$. Let $i \in Q$ be an involution. Then it is easy to see that $iQ$ is a finite uniquely $2$-divisible twisted subgroup in $G$. Therefore $\langle iQ \rangle$ is solvable. Moreover, $\Cen(i) \leq N_G(\langle iQ \rangle )$ and $G$ can be decomposed as $G = iQ\Cen(i)$. Therefore $\langle iQ \rangle$ is a solvable normal subgroup of $G$. It follows that $G$ contains a non-trivial abelian normal subgroup. Now \cref{prop:BHNeumann} implies that $Q$ consists of a single line.
\end{proof}

In the context of groups of finite Morley rank we do not know if every uniquely $2$-divisible K-loop of finite Morley rank can be definably half-embedded into a group of finite Morley rank.
The following would be a finite Morley rank version of Glauberman's theorem:

\begin{conjecture} \label{conj:twisted_subgroup}
Let $G$ be a connected group of finite Morley rank with a definable uniquely $2$-divisible twisted subgroup $L$ of Morley degree $1$ such that $G = \langle L \rangle$. Then $G$ is solvable.
\end{conjecture}

Note that this conjecture would imply the Feit-Thompson Theorem for connected groups of finite Morley rank: If $G$ is a connected group of finite Morley rank of degenerate type, then $G$ is uniquely $2$-divisible and hence \cref{conj:twisted_subgroup} (applied to $L = G$) would imply that $G$ is solvable.

Moreover, it would imply that Frobenius groups of finite Morley rank split: For Frobenius groups of degenerate type this would follow from solvability. If $G$ is a connected Frobenius group of finite Morley rank and odd type with involutions $J$ and lines $\Lambda$, then it suffices to show that $G$ has a non-trivial definable solvable normal subgroup (in that case $G$ has a non-trivial abelian normal subgroup and hence splits by \cref{prop:BHNeumann}). Note that $iJ$ is a uniquely $2$-divisible twisted subgroup. If $G$ is sharply $2$-transitive, then \cref{prop:subgroup generated by Q^2} shows that $\langle iJ\rangle$ is definable and connected and hence should be solvable by \cref{conj:twisted_subgroup}. 

For the general case consider the family $\calF_i = \{ \Cen(ij)^0 : j \in J \setminus \{i \} \}$. By Zilber's indecomposability theorem the subgroup $N = \langle H : H \in \calF_i \rangle$ is definable and connected. Moreover, it is easy to see that $N \cap iJ$ must be generic in $iJ$ and $N$ must be normalized by $\Cen(i)$. Therefore $N$ must be a normal subgroup of $G = iJ\Cen(i)$ and clearly $N = \langle N \cap iJ \rangle$. Hence \cref{conj:twisted_subgroup} would imply that $N$ is solvable.

If Frobenius groups of odd and degenerate type split, then \cref{rem:even-type} shows that Frobenius groups of even type also split.

If the twisted subgroup in the statement of \cref{conj:twisted_subgroup} is strongly minimal, then we show that $G$ must be $2$-nilpotent:

\begin{proposition} \label{prop:strongly-minimal-twisted}
Let $G$ be a connected group of finite Morley rank with a definable strongly minimal uniquely $2$-divisible twisted subgroup $L$ such that $G = \langle L \rangle$. Then $G$ is $2$-nilpotent.
\end{proposition}
\begin{proof}
Let $x \otimes y = x^\half y x^\half$ be the corresponding K-loop structure on $L$.
If $(L,\otimes)$ is an abelian group, then \cite[Theorem 6.14, part (3)]{kiechle} implies $[[a,b],c] = 1$ for all $a,b,c \in L$ and therefore $G = \langle L \rangle$ must be $2$-nilpotent.
Therefore it suffices to show that $(L,\otimes)$ is an abelian group.

Put $T = N_G(L)/\Cen(L)$. Then $T \leq \Aut((L,\otimes))$ and we may consider the quasidirect product $\calG = L \rtimes_Q T$.
As stated in \cref{prop:quasidirect_product} the group $\calG = L \rtimes_Q T$ acts transitively and faithfully on $L$ by
\[ (a,\alpha)(x) = a \otimes \alpha(x) \]
and $T$ is the stabilizer of $1 \in L$.
Note that $L' = L \times \{ 1 \}$ is a uniquely $2$-divisible twisted subgroup of $\calG$. Hence $a \otimes' b = a^\half b a^\half$ defines a K-loop structure on $L'$. By \cite[Theorem 6.15]{kiechle} the K-loops $(L,\otimes)$ and $(L', \otimes')$ are isomorphic. Therefore it suffices to show that $(L',\otimes')$ is an abelian group.

Hrushovski's analysis of groups acting on strongly minimal sets (Theorem 11.98 of \cite{borovik-nesin}) shows that $\MR(\calG) \leq 3$. Moreover, if $\MR(\calG) = 3$, then $T$ acts sharply $2$-transitively on $L \setminus \{1\}$ which is impossible since $T$ is a group of automorphisms of $(L,\otimes)$.

If $\MR(\calG) = 2$, then $L \rtimes_Q T$ is a standard sharply $2$-transitive group $K_+ \rtimes K^*$ (and the corresponding permutation groups coincide). Since $L'$ acts without fixed points and the fixed point free elements of $K_+ \rtimes K^*$ are precisely the elements of $K_+$, $L'$ is contained in $K_+$. Therefore $\otimes'$ agrees with the group structure on $K_+$ and hence $(L', \otimes')$ is an abelian group.

Now assume $\MR(\calG) = 1$. We argue similarly to the proof of \cite[Lemma 5, part (v)]{B-loops}.

Consider the finite twisted subgroup $L'' = \{ a\calG^0 : a \in L' \}$ of $\calG/\calG^0$. Since $L'$ is uniquely $2$-divisible the map $L'' \rightarrow L'', a \mapsto a^2$ is surjective and hence a bijection since $L''$ is finite. Hence we may define a K-loop structure $x\calG^0 \otimes'' y \calG^0 = x^\half yx^\half \calG^0$ on $L''$. The natural map $L' \rightarrow L''$ is a surjective homomorphism from $(L', \otimes')$ to $(L'', \otimes'')$ with kernel $L' \cap \calG^0$. 

In particular, $L' \cap \calG^0$ is a normal subloop of $L'$. Since $L'/(L' \cap \calG^0)$ is finite and $\Mdeg(L') = 1$ this implies $L' = L' \cap \calG^0$ and hence $L' \subseteq \calG^0$. The group $\calG^0$ is strongly minimal and thus abelian. Therefore $\otimes'$ agrees with the group structure on $\calG^0$ and therefore $(L', \otimes')$ is an abelian group.
\end{proof}

The proof of \cref{prop:strongly-minimal-twisted} in fact shows the following:
\begin{corollary} \label{cor:strongly-minimal-loop}
Let $G$ be a group of finite Morley rank and let $L \subseteq G$ be a definable uniquely $2$-divisible twisted subgroup of $G$.
\begin{enumerate}
\item If $\Mdeg(L) = 1$, then $L \subseteq G^0$.
\item If $L$ is strongly minimal, then the associated K-loop $(L,\otimes)$ is an abelian group and hence $\langle L \rangle$ is $2$-nilpotent (without assuming that $\langle L \rangle$ is definable).
\end{enumerate}
\end{corollary}
In particular, if $(L, \otimes)$ is a strongly minimal uniquely $2$-divisible K-loop such that $L$ can be definably half-embedded into a group of finite Morley rank, then $(L, \otimes)$ is an abelian group.

\begin{question}
This sugests the following two questions:
\begin{enumerate}
\item Suppose $G$ and $L$ satisfy the assumptions of \cref{prop:strongly-minimal-twisted}. Must $G$ be abelian?
\item Is every strongly minimal (uniquely $2$-divisible) K-loop an abelian group?
\end{enumerate}
\end{question}

\bibliographystyle{plain}
\bibliography{bibliography}

\end{document}